\documentclass[a4paper, british]{amsart}

\usepackage{libertine}
\usepackage[libertine]{newtxmath}

\usepackage{amssymb}
\usepackage{babel}
\usepackage{enumitem}
\usepackage{hyperref}
\usepackage[utf8]{inputenc}
\usepackage{newunicodechar}
\usepackage{mathtools}
\usepackage{varioref}
\usepackage[arrow,curve,matrix]{xy}

\usepackage{colortbl}
\usepackage{graphicx}
\usepackage{tikz}

\usepackage{mathrsfs}

\definecolor{linkred}{rgb}{0.7,0.2,0.2}
\definecolor{linkblue}{rgb}{0,0.2,0.6}

\numberwithin{figure}{section}

\usepackage[hyperpageref]{backref}

\setdescription{labelindent=\parindent, leftmargin=2\parindent}
\setitemize[1]{labelindent=\parindent, leftmargin=2\parindent}
\setenumerate[1]{labelindent=0cm, leftmargin=*, widest=iiii}

\newunicodechar{α}{\ensuremath{\alpha}}
\newunicodechar{β}{\ensuremath{\beta}}
\newunicodechar{χ}{\ensuremath{\chi}}
\newunicodechar{δ}{\ensuremath{\delta}}
\newunicodechar{ε}{\ensuremath{\varepsilon}}
\newunicodechar{Δ}{\ensuremath{\Delta}}
\newunicodechar{η}{\ensuremath{\eta}}
\newunicodechar{γ}{\ensuremath{\gamma}}
\newunicodechar{Γ}{\ensuremath{\Gamma}}
\newunicodechar{ι}{\ensuremath{\iota}}
\newunicodechar{κ}{\ensuremath{\kappa}}
\newunicodechar{λ}{\ensuremath{\lambda}}
\newunicodechar{Λ}{\ensuremath{\Lambda}}
\newunicodechar{ν}{\ensuremath{\nu}}
\newunicodechar{μ}{\ensuremath{\mu}}
\newunicodechar{ω}{\ensuremath{\omega}}
\newunicodechar{Ω}{\ensuremath{\Omega}}
\newunicodechar{π}{\ensuremath{\pi}}
\newunicodechar{Π}{\ensuremath{\Pi}}
\newunicodechar{φ}{\ensuremath{\phi}}
\newunicodechar{Φ}{\ensuremath{\Phi}}
\newunicodechar{ψ}{\ensuremath{\psi}}
\newunicodechar{Ψ}{\ensuremath{\Psi}}
\newunicodechar{ρ}{\ensuremath{\rho}}
\newunicodechar{σ}{\ensuremath{\sigma}}
\newunicodechar{Σ}{\ensuremath{\Sigma}}
\newunicodechar{τ}{\ensuremath{\tau}}
\newunicodechar{θ}{\ensuremath{\theta}}
\newunicodechar{Θ}{\ensuremath{\Theta}}
\newunicodechar{ξ}{\ensuremath{\xi}}
\newunicodechar{Ξ}{\ensuremath{\Xi}}
\newunicodechar{ζ}{\ensuremath{\zeta}}

\newunicodechar{ℓ}{\ensuremath{\ell}}

\newunicodechar{𝔹}{\ensuremath{\bB}}
\newunicodechar{ℂ}{\ensuremath{\bC}}
\newunicodechar{𝔻}{\ensuremath{\bD}}
\newunicodechar{𝔼}{\ensuremath{\bE}}
\newunicodechar{𝔽}{\ensuremath{\bF}}
\newunicodechar{ℕ}{\ensuremath{\bN}}
\newunicodechar{ℙ}{\ensuremath{\bP}}
\newunicodechar{ℚ}{\ensuremath{\bQ}}
\newunicodechar{ℝ}{\ensuremath{\bR}}
\newunicodechar{𝕏}{\ensuremath{\bX}}
\newunicodechar{ℤ}{\ensuremath{\bZ}}
\newunicodechar{𝒜}{\ensuremath{\sA}}
\newunicodechar{ℬ}{\ensuremath{\sB}}
\newunicodechar{𝒞}{\ensuremath{\sC}}
\newunicodechar{𝒟}{\ensuremath{\sD}}
\newunicodechar{ℰ}{\ensuremath{\sE}}
\newunicodechar{ℱ}{\ensuremath{\sF}}
\newunicodechar{𝒢}{\ensuremath{\sG}}
\newunicodechar{ℋ}{\ensuremath{\sH}}
\newunicodechar{𝒥}{\ensuremath{\sJ}}
\newunicodechar{ℒ}{\ensuremath{\sL}}
\newunicodechar{𝒪}{\ensuremath{\sO}}
\newunicodechar{𝒬}{\ensuremath{\sQ}}
\newunicodechar{𝒯}{\ensuremath{\sT}}
\newunicodechar{𝒲}{\ensuremath{\sW}}

\newunicodechar{∂}{\ensuremath{\partial}}
\newunicodechar{∇}{\ensuremath{\nabla}}

\newunicodechar{↺}{\ensuremath{\circlearrowleft}}
\newunicodechar{∞}{\ensuremath{\infty}}
\newunicodechar{⊕}{\ensuremath{\oplus}}
\newunicodechar{⊗}{\ensuremath{\otimes}}
\newunicodechar{•}{\ensuremath{\bullet}}
\newunicodechar{Λ}{\ensuremath{\wedge}}
\newunicodechar{↪}{\ensuremath{\into}}
\newunicodechar{→}{\ensuremath{\to}}
\newunicodechar{↦}{\ensuremath{\mapsto}}
\newunicodechar{⨯}{\ensuremath{\times}}
\newunicodechar{∪}{\ensuremath{\cup}}
\newunicodechar{∩}{\ensuremath{\cap}}
\newunicodechar{⊋}{\ensuremath{\supsetneq}}
\newunicodechar{⊇}{\ensuremath{\supseteq}}
\newunicodechar{⊃}{\ensuremath{\supset}}
\newunicodechar{⊊}{\ensuremath{\subsetneq}}
\newunicodechar{⊆}{\ensuremath{\subseteq}}
\newunicodechar{⊂}{\ensuremath{\subset}}
\newunicodechar{≥}{\ensuremath{\geq}}
\newunicodechar{≠}{\ensuremath{\neq}}
\newunicodechar{≫}{\ensuremath{\gg}}
\newunicodechar{≪}{\ensuremath{\ll}}

\newunicodechar{≤}{\ensuremath{\leq}}
\newunicodechar{∈}{\ensuremath{\in}}
\newunicodechar{∉}{\ensuremath{\not \in}}
\newunicodechar{∖}{\ensuremath{\setminus}}
\newunicodechar{◦}{\ensuremath{\circ}}
\newunicodechar{°}{\ensuremath{^\circ}}
\newunicodechar{…}{\ifmmode\mathellipsis\else\textellipsis\fi}
\newunicodechar{·}{\ensuremath{\cdot}}
\newunicodechar{⋯}{\ensuremath{\cdots}}
\newunicodechar{∅}{\ensuremath{\emptyset}}
\newunicodechar{⇒}{\ensuremath{\Rightarrow}}

\newunicodechar{⁰}{\ensuremath{^0}}
\newunicodechar{¹}{\ensuremath{^1}}
\newunicodechar{²}{\ensuremath{^2}}
\newunicodechar{³}{\ensuremath{^3}}
\newunicodechar{⁴}{\ensuremath{^4}}
\newunicodechar{⁵}{\ensuremath{^5}}
\newunicodechar{⁶}{\ensuremath{^6}}
\newunicodechar{⁷}{\ensuremath{^7}}
\newunicodechar{⁸}{\ensuremath{^8}}
\newunicodechar{⁹}{\ensuremath{^9}}
\newunicodechar{ⁱ}{\ensuremath{^i}}

\newunicodechar{⌈}{\ensuremath{\lceil}}
\newunicodechar{⌉}{\ensuremath{\rceil}}
\newunicodechar{⌊}{\ensuremath{\lfloor}}
\newunicodechar{⌋}{\ensuremath{\rfloor}}

\newunicodechar{≅}{\ensuremath{\cong}}
\newunicodechar{⇔}{\ensuremath{\Leftrightarrow}}

\DeclareFontFamily{OMS}{rsfs}{\skewchar\font'60}
\DeclareFontShape{OMS}{rsfs}{m}{n}{<-5>rsfs5 <5-7>rsfs7 <7->rsfs10 }{}
\DeclareSymbolFont{rsfs}{OMS}{rsfs}{m}{n}
\DeclareSymbolFontAlphabet{\scr}{rsfs}
\DeclareSymbolFontAlphabet{\scr}{rsfs}

\DeclareFontFamily{U}{mathx}{\hyphenchar\font45}
\DeclareFontShape{U}{mathx}{m}{n}{
      <5> <6> <7> <8> <9> <10>
      <10.95> <12> <14.4> <17.28> <20.74> <24.88>
      mathx10
      }{}
\DeclareSymbolFont{mathx}{U}{mathx}{m}{n}
\DeclareFontSubstitution{U}{mathx}{m}{n}
\DeclareMathAccent{\wcheck}{0}{mathx}{"71}

\DeclareMathOperator{\Aut}{Aut}
\DeclareMathOperator{\codim}{codim}

\DeclareMathOperator{\Hom}{Hom}
\DeclareMathOperator{\Id}{Id}
\DeclareMathOperator{\Image}{Image}
\DeclareMathOperator{\img}{img}
\DeclareMathOperator{\Pic}{Pic}
\DeclareMathOperator{\rank}{rank}

\DeclareMathOperator{\red}{red}
\DeclareMathOperator{\reg}{reg}

\DeclareMathOperator{\Sym}{Sym}
\DeclareMathOperator{\supp}{supp}
\DeclareMathOperator{\tor}{tor}

\newcommand{\sA}{\scr{A}}
\newcommand{\sB}{\scr{B}}
\newcommand{\sC}{\scr{C}}
\newcommand{\sD}{\scr{D}}
\newcommand{\sE}{\scr{E}}
\newcommand{\sF}{\scr{F}}
\newcommand{\sG}{\scr{G}}
\newcommand{\sH}{\scr{H}}
\newcommand{\sHom}{\scr{H}\negthinspace om}

\newcommand{\sJ}{\scr{J}}

\newcommand{\sL}{\scr{L}}
\newcommand{\sM}{\scr{M}}

\newcommand{\sO}{\scr{O}}

\newcommand{\sQ}{\scr{Q}}

\newcommand{\sT}{\scr{T}}

\newcommand{\sV}{\scr{V}}
\newcommand{\sW}{\scr{W}}

\newcommand{\cC}{\mathcal C}

\newcommand{\cE}{\mathcal E}

\newcommand{\bB}{\mathbb{B}}
\newcommand{\bC}{\mathbb{C}}
\newcommand{\bD}{\mathbb{D}}
\newcommand{\bE}{\mathbb{E}}
\newcommand{\bF}{\mathbb{F}}

\newcommand{\bN}{\mathbb{N}}

\newcommand{\bP}{\mathbb{P}}
\newcommand{\bQ}{\mathbb{Q}}
\newcommand{\bR}{\mathbb{R}}

\newcommand{\bX}{\mathbb{X}}

\newcommand{\bZ}{\mathbb{Z}}

\newcommand{\aB}{{\sf B}}
\newcommand{\aD}{{\sf D}}
\newcommand{\aE}{{\sf E}}
\newcommand{\aF}{{\sf F}}

\theoremstyle{plain}
\newtheorem{thm}{Theorem}[section]

\newtheorem{cor}[thm]{Corollary}
\newtheorem{defn}[thm]{Definition}
\newtheorem{fact}[thm]{Fact}
\newtheorem{lem}[thm]{Lemma}

\newtheorem{prop}[thm]{Proposition}

\theoremstyle{remark}
\newtheorem{assumption}[thm]{Assumption}
\newtheorem{asswlog}[thm]{Assumption w.l.o.g.}
\newtheorem{claim}[thm]{Claim}
\newtheorem{c-n-d}[thm]{Claim and Definition}

\newtheorem{construction}[thm]{Construction}

\newtheorem{notation}[thm]{Notation}
\newtheorem{obs}[thm]{Observation}
\newtheorem{rem}[thm]{Remark}

\newtheorem*{rem-nonumber}{Remark}

\newtheorem{warning}[thm]{Warning}

\numberwithin{equation}{thm}

\setlist[enumerate]{label=(\thethm.\arabic*), before={\setcounter{enumi}{\value{equation}}}, after={\setcounter{equation}{\value{enumi}}}}

\newcommand{\into}{\hookrightarrow}

\newcommand{\wtilde}{\widetilde}
\newcommand{\what}{\widehat}

\newcommand\CounterStep{\addtocounter{thm}{1}\setcounter{equation}{0}}

\newcommand{\factor}[2]{\left. \raise 2pt\hbox{$#1$} \right/\hskip -2pt\raise -2pt\hbox{$#2$}}
\newcommand{\Preprint}[1]{}
\newcommand{\Publication}[1]{#1}

\newcommand{\subversionInfo}{}
\newcommand{\svnid}[1]{}
\newcommand{\approvals}[2][Approval]{}

\renewcommand{\phi}{\varphi}
\allowdisplaybreaks

\author{Daniel Greb} %
\address{Daniel Greb, Essener Seminar für Algebraische Geometrie und Arithmetik,
  Fakultät für Ma\-the\-matik, Universität Duisburg--Essen, 45117 Essen,
  Germany} %
\email{\href{mailto:daniel.greb@uni-due.de}{daniel.greb@uni-due.de}}
\urladdr{\url{http://www.esaga.uni-due.de/daniel.greb}}

\author{Stefan Kebekus} %
\address{Stefan Kebekus, Mathematisches Institut, Albert-Ludwigs-Universität
  Freiburg, Ernst-Zermelo-Straße 1, 79104 Freiburg im Breisgau, Germany \&
  Freiburg Institute for Advanced Studies (FRIAS), Freiburg im Breisgau,
  Germany} %
\email{\href{mailto:stefan.kebekus@math.uni-freiburg.de}{stefan.kebekus@math.uni-freiburg.de}}
\urladdr{\url{https://cplx.vm.uni-freiburg.de}}

\author{Thomas Peternell} %
\address{Thomas Peternell, Mathematisches Institut, Universität Bayreuth,
  95440~Bayreuth, Germany} %
\email{\href{mailto:thomas.peternell@uni-bayreuth.de}{thomas.peternell@uni-bayreuth.de}}
\urladdr{\url{http://www.komplexe-analysis.uni-bayreuth.de}}

\author{Behrouz Taji} %
\address{Behrouz Taji, University of Notre Dame, Department of Mathematics, 278
  Hurley, Notre Dame, IN 46556 USA.} %
\email{\href{mailto:btaji@nd.edu}{btaji@nd.edu}}
\urladdr{\url{http://sites.nd.edu/b-taji}}

\thanks{DG was partially supported by the DFG-Collaborative Research Center
  SFB/TR 45.  SK gratefully acknowledges support through a joint fellowship of
  the Freiburg Institute of Advanced Studies (FRIAS) and the University of
  Strasbourg Institute for Advanced Study (USIAS).  BT was partially supported
  by the DFG-Research Training Group GK1821.  Research was partially completed
  while SK and TP were visiting the National University of Singapore in 2017.  TP
  was supported by the DFG project ``Zur Positivität in der komplexen
  Geometrie''.}



\title[NAHT for klt spaces and descent]{Nonabelian Hodge Theory for klt spaces and descent theorems for vector bundles}
\date{\today}

\makeatletter
\hypersetup{
  pdfauthor={\authors},
  pdftitle={\@title},
  pdfsubject={\@subjclass},
  pdfkeywords={\@keywords},
  pdfstartview={Fit},
  pdfpagelayout={TwoColumnRight},
  pdfpagemode={UseOutlines},
  bookmarks,
  colorlinks,
  linkcolor=linkblue,
  citecolor=linkred,
  urlcolor=linkred}
\makeatother

\DeclareMathOperator{\Div}{Div}

\DeclareMathOperator{\Exc}{Exc}

\DeclareMathOperator{\Higgs}{\sf Higgs}
\DeclareMathOperator{\HN}{HN}

\DeclareMathOperator{\LSys}{\sf LSys}

\DeclareMathOperator{\refl}{refl}

\theoremstyle{plain}

\theoremstyle{remark}
\newtheorem*{note}{Note}
\newtheorem{choice}[thm]{Choice}

\newtheorem{reminder}[thm]{Reminder}

\begin{document}
\begin{abstract}
We generalise Simpson's nonabelian Hodge correspondence to the context of
projective varieties with klt singularities.  The proof relies on a descent
theorem for numerically flat vector bundles along birational morphisms.  In its
simplest form, this theorem asserts that given any klt variety $X$ and any
resolution of singularities, then any vector bundle on the resolution that
appears to come from $X$ numerically, does indeed come from $X$.  Furthermore
and of independent interest, a new restriction theorem for semistable Higgs
sheaves defined on the smooth locus of a normal, projective variety is
established.
\end{abstract}
\approvals[Abstract]{Behrouz & yes \\ Daniel & yes \\ Stefan & yes \\ Thomas & yes}

\maketitle
\tableofcontents

%
% Do not edit the following line.  The text is automatically updated by
% subversion.
%
\svnid{$Id: 01.tex 1181 2018-10-26 00:45:35Z taji $}

\section{Introduction}
\subversionInfo

\subsection{Nonabelian Hodge theory for singular spaces}
\approvals{Behrouz & yes \\ Daniel & yes \\ Stefan & yes \\ Thomas & yes}

Given a projective manifold $X$, a seminal result of Simpson, \cite{MR1179076},
exhibits a natural equivalence between the category of local systems and the
category of semistable, locally free Higgs sheaves with vanishing Chern classes.
The first main result of this paper extends this correspondence to projective
varieties with Kawamata log terminal (=klt) singularities.

\begin{thm}[Nonabelian Hodge correspondence for klt spaces, Theorem~\ref{thm:free-naht}]\label{thm:main}
  Let $X$ be a projective, klt variety.  Then, there exists an equivalence
  between the category of local systems and the category of semistable, locally
  free Higgs sheaves with vanishing Chern classes.  \qed
\end{thm}

We refer the reader to \cite[Sect.~5]{GKPT15} or to the survey
\cite[Sect.~6]{GKT16} for the (rather delicate) notions of Higgs sheaves,
morphisms of Higgs sheaves and pull-back.  Semistability is also discussed
there.

In fact, there exists a unique way to choose the correspondences of
Theorem~\ref{thm:main} that makes them functorial in morphisms between klt
spaces, and compatible with Simpson's construction wherever this makes sense; we
refer to Section~\vref{sect:NAHT} for a precise formulation.  In particular,
functoriality applies to resolutions of singularities, as well as morphisms
whose images are entirely contained in the singular loci of the target
varieties.  Our results imply that the pull-back of semistable, locally free
Higgs sheaves with vanishing Chern classes under any of these maps remains
semistable.

\subsection{Descent of vector bundles}\label{ssec:introdesc}
\approvals{Behrouz & yes \\ Daniel & yes \\ Stefan & yes \\ Thomas & yes}

The proof of Theorem~\ref{thm:main} relies in part on a descent theorem for
vector bundles which is of independent interest.  To put the result into
perspective, consider a desingularisation $π: \wtilde{X} → X$ of a normal
variety.  It is well-known that if $X$ has rational singularities, then any line
bundle on $\wtilde{X}$ that comes from $X$ topologically does in fact come from
$X$ holomorphically.  If $X$ is klt, we generalise this result to vector bundles
of arbitrary rank: we show that any vector bundle on $\wtilde{X}$ that appears
to come from $X$ numerically does indeed come from $X$.

\begin{thm}[Descent of vector bundles to klt spaces, Theorem~\ref{thm:desc2}]\label{thm:desc2simplified}
  Let $f : X → Y$ be a birational, projective morphism of normal,
  quasi-projective varieties.  Assume that there exists a Weil $ℚ$-divisor $Δ_Y$
  such that $(Y,Δ_Y)$ is klt.  If $ℱ_X$ is any locally free, $f$-numerically
  flat sheaf on $X$, then there exists a locally free sheaf $ℱ_Y$ on $Y$ such
  that $ℱ_X ≅ f^* ℱ_Y$.  \qed
\end{thm}

The notion of ``numerical flatness for vector bundles'' generalises the notion
of ``numerical triviality'' for line bundles and is recalled in
Definition~\ref{def:numflat} below.  The importance of
Theorem~\ref{thm:desc2simplified} in the current investigation stems from the
fact that locally free Higgs sheaves coming from local systems on resolutions of
klt spaces are numerically flat relative to the resolution morphism; see
Proposition~\ref{prop:Higgs_are_num_flat}.

In addition to the descent result for vector bundles spelled out above,
Theorem~\ref{thm:desc2} also discusses descent of locally free Higgs sheaves.
Theorem~\ref{thm:desc1} contains a related result where $X$ (rather than $Y$) is
assumed to be klt.

\subsection{Optimality of results}
\approvals{Behrouz & yes \\ Daniel & yes \\ Stefan & yes \\ Thomas & yes}

We expect that varieties with klt singularities form the largest natural class
where our results can possibly hold in full generality.  The construction of a
pull-back functor for Higgs sheaves is rather delicate and hinges on the
existence of functorial pull-back for reflexive differentials, for morphisms
between klt spaces.  For classes of varieties with singularities that are just
slightly more general than klt, there are elementary examples,
\cite[Ex.~1.9]{MR3084424}, which show that \emph{functorial} pull-back for
reflexive differentials does not exist, even though it is known that reflexive
differentials still lift to resolutions of singularities in these cases,
\cite[Thm.~1.4]{MR3247804}.  In particular, it is \emph{not possible} to define
functorial pull-back of Higgs sheaves for these spaces.

For spaces with arbitrary singularities, we do not expect that a correspondence
between the two categories in Theorem~\ref{thm:main} holds, even at the level of
objects.

\subsection{Applications}
\approvals{Behrouz & yes \\ Daniel & yes \\ Stefan & yes \\ Thomas & yes}

Theorem~\ref{thm:main} has applications to the quasi-étale uniformisation
problem for minimal varieties of general type.  Eventually, we expect that all
uniformisation theorems of nonabelian Hodge theory have analogues for klt
varieties.  In particular we expect to generalise the uniformisation result
\cite[Thm.~1.2]{GKPT15} to arbitrary klt varieties: if $X$ is minimal, klt and
of general type, and if equality holds in the $ℚ$-Miyaoka-Yau inequality,
\cite[Thm.~1.1]{GKPT15}, i.e.,
$$
\bigl( 2(n+1)· \what{c}_2(𝒯_X) - n · \what{c}_1(𝒯_X)² \bigr) · [K_X]^{n-2} =0,
$$
then the canonical model of $X$ is a singular ball quotient.

To keep the current paper reasonably short, these applications will appear in a
separate paper, cf.~\cite{GKPT18}.  Please see the survey paper \cite{GKT16} for
related results in this direction.

One might expect similar uniformisation results in more general contexts, such
as ``pairs'' or ``orbifolds''.  In the setting of pairs, the Miyaoka-Yau
inequality has already been established in~\cite{GT16}.  We refer to
\cite[Sect.~10]{GKPT15} for a more precise formulation.

\subsection{Relation to the work of Mochizuki}
\approvals{Behrouz & yes \\ Daniel & yes \\ Stefan & yes \\ Thomas & yes}

In a large body of work, T.~Mochizuki set up a complete theory of Higgs bundles
on arbitrary smooth quasi-projective varieties; among others, see
\cite{MR2310103, MR2281877, MR2283665}.  Our study differs from Mochizuki’s in
at least two aspects: First, while our main result, Theorem~\ref{thm:main}
above, traces a correspondence between topological and algebro-geometric
properties of $X$, thereby taking the singularities of $X$ into account, the
correspondence established by Mochizuki in the present setup would focus on the
connection between such properties for the smooth locus $X_{\reg}$.  Second,
ultimately our correspondence is induced geometrically from the nonabelian Hodge
correspondence in the projective case, which in turn requires much less
sophisticated analytic results than Mochizuki's approach.  However, Mochizuki's
theory will be used in the sequel paper \cite{GKPT18}.

\subsection{Structure of the paper}
\approvals{Behrouz & yes \\ Daniel & yes \\ Stefan & yes \\ Thomas & yes}

Section~\ref{sect:Notations} gathers notation, known results and global
conventions that will be used throughout the paper.  Section~\ref{sect:NAHT}
formulates the nonabelian Hodge correspondence for klt spaces in detail,
discusses functoriality and its consequence, and states a number of singular
generalisations of Simpson's classical results.

The results are then proven in the remaining sections.  Sections~\ref{sec:03}
and \ref{sec:4} prepare for the proof, establishing the descent theorems for
vector bundles mentioned in \ref{ssec:introdesc} above.
Section~\ref{sect:app-oper} establishes a restriction theorem of
Mehta-Ramanathan type for Higgs sheaves that is slightly more general than the
versions found in the literature.  This restriction theorem is then used in
Section~\ref{ssec:ascend} to prove that Higgs bundles with vanishing Chern
classes on projective, klt varieties that are semistable with respect to an
ample divisor, remain semistable with respect to \emph{ample} divisors when
pulled back to a resolution of singularities.  With these preparations in place,
the nonabelian Hodge correspondences can then be constructed in
Section~\ref{sect:07} without much effort.

\subsection{Acknowledgements}
\approvals{Behrouz & yes \\ Daniel & yes \\ Stefan & yes \\ Thomas & yes}

We would like to thank numerous colleagues for discussions, including Daniel
Barlet, Oliver Bräunling, Andreas Höring, Annette Huber, Shane Kelly, Jong-Hae
Keum, Adrian Langer, and Jörg Schürmann.  We thank the referee for carefully
reading the paper, and for her or his very valuable and interesting remarks.

%
% Do not edit the following line.  The text is automatically updated by
% subversion.
%
\svnid{$Id: 02.tex 1180 2018-10-25 09:04:48Z kebekus $}

\section{Notation and elementary facts}\label{sect:Notations}
\subversionInfo

\subsection{Global conventions}
\approvals{Behrouz & yes \\ Daniel & yes \\ Stefan & yes \\ Thomas & yes}

Throughout the present paper, all varieties and schemes will be defined over the
complex numbers.  We will freely switch between the algebraic and analytic
context if no confusion is likely to arise.  Apart from that, we follow the
notation used in the standard reference books \cite{Ha77, KM98}, with the
exception that for klt pairs $(X,Δ)$, the boundary divisor $Δ$ is always assumed
to be effective.  Varieties are always assumed to be irreducible and reduced.

\subsection{Varieties, sets and morphisms}
\approvals{Behrouz & yes \\ Daniel & yes \\ Stefan & yes \\ Thomas & yes}

Normal varieties are $S_2$, which implies that regular functions can be extended
across sets of codimension two.  The following notation will be used.

\begin{notation}[Big and small subsets]
  Let $X$ be a normal, quasi-projective variety.  A Zariski-closed subset
  $Z ⊂ X$ is called \emph{small} if $\codim_X Z ≥ 2$.  A Zariski-open subset
  $U ⊆ X$ is called \emph{big} if $X∖ U$ is small.  A birational morphism
  $φ : X → Y$ of normal, projective varieties is called a \emph{small morphism}
  if there exists a big open set $X° ⊆ X$ such that $φ|_{X°}$ is an open
  immersion.
\end{notation}

\begin{notation}[Set-theoretic fibre]
  Given a morphism of varieties $φ: X → Y$ and a point $y ∈ Y$, we call the
  reduced scheme $φ^{-1}(y)_{\red}$ the \emph{set-theoretic fibre} of $φ$ over
  $y$.
\end{notation}

\begin{defn}[Covers and covering maps, Galois morphisms]\label{def:cover}
  A \emph{cover} or \emph{covering map} is a finite, surjective morphism
  $γ : Y → X$ of normal, quasi-projective varieties or complex spaces.  The
  covering map $γ$ is called \emph{Galois} if there exists a finite group
  $G ⊆ \Aut(Y)$ such that $X$ is isomorphic to the quotient map $Y → Y/G$.
\end{defn}

\begin{defn}[Quasi-étale morphisms]\label{defn:quasietale}
  A morphism $f : X → Y$ between normal varieties is called \emph{quasi-étale}
  if $f$ is of relative dimension zero and étale in codimension one.  In other
  words, $f$ is quasi-étale if $\dim X = \dim Y$ and if there exists a closed,
  subset $Z ⊆ X$ of codimension $\codim_X Z ≥ 2$ such that
  $f|_{X ∖ Z} : X ∖ Z → Y$ is étale.
\end{defn}

\begin{note}
  Combining Definitions~\ref{def:cover} and \ref{defn:quasietale}, a quasi-étale
  cover is finite, surjective, and étale in codimension one.
\end{note}

\subsection{Line bundles and sheaves}
\approvals{Behrouz & yes \\ Daniel & yes \\ Stefan & yes \\ Thomas & yes}

Reflexive sheaves are in many ways easier to handle than arbitrary coherent
sheaves, and we will therefore frequently take reflexive hulls.  The following
notation will be used.

\begin{notation}[Reflexive hull]
  Given a normal, quasi-projective variety $X$ and a coherent sheaf $ℰ$ on $X$
  of rank $r$, write
  $$
  Ω^{[p]}_X := \bigl(Ω^p_X \bigr)^{**}, \quad ℰ^{[m]} := \bigl(ℰ^{⊗ m}
  \bigr)^{**}, \quad \Sym^{[m]} ℰ := \bigl( \Sym^m ℰ \bigr)^{**}
  $$
  and $\det ℰ := \bigl( Λ^r ℰ \bigr)^{**}$.  Given any morphism $f : Y → X$,
  write $f^{[*]} ℰ := (f^* ℰ)^{**}$.
\end{notation}

\begin{defn}[Relative Picard number]
  Given a projective surjection $f: X → Y$ of normal, quasi-projective
  varieties, let $N_1(X/Y)$ be the $ℝ$-vector space generated by irreducible
  curves $C ⊆ X$ such that $f(C)$ is a point, modulo numerical equivalence.  The
  dimension of $N_1(X/Y)$ is called the relative Picard number of $X/Y$ and is
  denoted by $ρ(X/Y)$.  Let $\overline{NE}(X/Y) ⊆ N_1(X/Y)$ be the closed cone
  generated by classes of effective curves which are contracted by $f$.
\end{defn}

\subsection{Cycles}
\approvals{Behrouz & yes \\ Daniel & yes \\ Stefan & yes \\ Thomas & yes}

The Chow variety represents a functor.  The associated notion of ``families of
cycles'' is however somewhat awkward to formulate.  For the sake of simplicity,
we restrict ourselves to families over normal base varieties where the
definition becomes somewhat simpler.  The book \cite{MR3307241} discusses these
matters in detail.

\begin{notation}[Families of cycles]\label{not:ctp}
  Let $X$ be a quasi-projective, $n$-dimensional variety, not necessarily
  normal.  Given any subscheme $Y ⊆ X$, we denote the associated cycle by $[Y]$.
  Let $f : X → Y$ be an equidimensional morphism of normal, algebraic varieties,
  of relative dimension $d$.  Recall from \cite[I~Thm.~3.17]{K96} or
  \cite[Thm.~3.4.1 on p.~449]{MR3307241} that $f$ is then a well defined family
  of $d$-dimensional proper algebraic cycles over $Y$, in the sense of
  \cite[I~Def.~3.10]{K96}.  In particular, if $y ∈ Y$ is any closed point,
  Kollár defines in \cite[I~Def.~3.10.4]{K96} the associated cycle-theoretic
  fibre, which we denote by $f^{[-1]}(y)$.
\end{notation}

\begin{warning}
  In the setting of Notation~\ref{not:ctp}, it is generally \emph{not} true that
  the cycle-theoretic fibre $f^{[-1]}(y)$ is the cycle associated with the
  scheme-theoretic fibre $f^{-1}(y)$.  Using the notation introduced above, the
  cycles $f^{[-1]}(y)$ and $[f^{-1}(y)]$ do not agree in general.
  \Publication{An example is discussed in the preprint version of this
    paper.}\Preprint{For a simple example, consider the morphism
    $$
    f:ℂ² → \bigl\{ (a,b,c) ∈ ℂ³ \,|\ ac=b² \bigr\}, \quad (x,y) ↦ \bigl(
    x²,xy,y² \bigr),
    $$
    which is the quotient morphism for the action of $ℤ_2 = \{ \pm 1\}$ on $ℂ²$
    by homotheties.  One computes directly from the definition that
    $$
    \bigl[f^{-1} \bigl(\vec 0_{ℂ³} \bigr)\bigr] = 3·\vec 0_{ℂ²}
    \quad\text{and}\quad f^{[-1]}\bigl(\vec 0_{ℂ³}\bigr) = 2·\vec 0_{ℂ²}
    $$}
\end{warning}

\begin{reminder}[Pull-back of Weil divisors]
  For arbitrary morphisms $f: X → Y$ normal, projective varieties, there is
  generally no good notion of pull-back for Weil-divisors.  If $f$ is finite, or
  more generally equidimensional, then a good pull-back map exists.  We refer to
  \cite[Sect.~2.4.1]{CKT16} for a discussion.
\end{reminder}

\subsection{Numerical classes}
\approvals{Behrouz & yes \\ Daniel & yes \\ Stefan & yes \\ Thomas & yes}

We briefly fix our notation for numerical classes and intersection numbers.  The
following definition allows to discuss slope and stability for arbitrary sheaves
on arbitrarily singular spaces.  We refer to \cite[Sect.~4.1]{GKP13} for a more
detailed discussion.

\begin{notation}[Numerical classes and intersection numbers]
  Let $X$ be a projective variety.  If $ℒ ∈ \Pic(X)$ is invertible and
  $D ∈ \Div(X)$ a Cartier divisor, we denote the numerical classes in
  $N¹(X)_{ℝ}$ by $[ℒ]$ and $[D]$; see \cite[Sect.~II.4]{K96} for a brief
  definition and discussion of the space $N¹(X)_{ℝ}$ of numerical Cartier
  divisor classes.  If $A$ is any purely $d$-dimensional cycle on $X$, we denote
  the intersection number with numerical classes of invertible sheaves $ℒ_i$ by
  $[ℒ_1]⋯[ℒ_d]·A ∈ ℤ$.  For brevity of notation we will often write
  $[ℒ_1]⋯[ℒ_n] ∈ ℤ$ instead of the more correct $[ℒ_1]⋯[ℒ_n]·[X]$.  Ditto for
  intersection with Cartier divisors.
\end{notation}

For the reader's convenience, we recall the notion of ``numerical flatness for
vector bundles'', which generalises the notion of ``numerical triviality'' for
line bundles.

\begin{defn}[\protect{Nefness and numerical flatness, \cite[Def.~1.17]{DPS94}}]\label{def:numflat}
  Let $φ : X → Y$ be a projective morphism of quasi-projective varieties.  Given
  a locally free sheaf $ℱ$ on $X$, consider the composed morphism
  $$
  \xymatrix{ %
    ℙ(ℱ) \ar[r] \ar@/^3mm/[rr]^{δ} & X \ar[r]_{φ} & Y.  %
  }
  $$
  The sheaf $ℱ$ is called \emph{$φ$-nef} if $[𝒪_{ℙ(ℱ)}(1)]$ is $δ$-nef, that is
  $[𝒪_{ℙ(ℱ)}(1)]·C ≥ 0$, for every irreducible curve $C ⊂ ℙ(ℱ)$ such that $δ(C)$
  is a point.  The sheaf $ℱ$ is called \emph{$φ$-numerically flat} if $ℱ$ and
  its dual $ℱ^*$ are both $φ$-nef.  If $Y$ is a point, we simply say that $ℱ$ is
  \emph{numerically flat}.
\end{defn}

\begin{rem}[Alternate formulations of numerical flatness]\label{rem:nflat}
  Setting as in Definition~\ref{def:numflat}.  The following conditions are
  equivalent.
  \begin{enumerate}
  \item The bundle $ℱ$ and its dual $ℱ^*$ are both $φ$-nef.
  \item The bundle $ℱ$ is $φ$-nef and the invertible sheaf $(\det ℱ)^*$ is
    $φ$-nef.
  \item The bundle $ℱ$ is $φ$-nef and the invertible sheaf $\det ℱ$ is
    $φ$-numerically trivial.
  \end{enumerate}
\end{rem}

Chern classes of numerically flat bundles vanish.  The proof is based on two
deep facts: the Uhlenbeck-Yau theorem asserting the existence of
Hermite-Einstein metrics on stable vector bundles, and the resulting
Kobayashi-Lübke flatness criterion derived from Lübke’s inequality on Chern
classes of Hermite-Einstein vector bundles.

\begin{thm}[\protect{Chern classes of numerically flat bundles, \cite[Cor.~1.19]{DPS94}}]\label{thm:numflat}
  Let $X$ be a smooth, projective variety and $ℱ$ a numerically flat, locally
  free sheaf on $X$.  Then, all Chern classes
  $c_i(ℱ) ∈ H^{2i}\bigl(X,\, ℝ \bigr)$ vanish.  \qed
\end{thm}

\subsection{KLT spaces, the Basepoint-free Theorem and contractions}
\approvals{Behrouz & yes \\ Daniel & yes \\ Stefan & yes \\ Thomas & yes}

We will mostly work with klt pairs $(X,Δ)$, but the boundary divisor $Δ$ is
usually irrelevant in our discussion.  We will use the following shorthand
notation throughout.  We refer to \cite{KM98} for a standard reference
concerning klt pairs but recall the global convention that the boundary divisor
is always assumed to be effective in this paper.

\begin{defn}[KLT spaces]\label{def:klt}
  A normal, quasi-projective variety is called a \emph{klt space} if there
  exists an effective Weil $ℚ$-divisor $Δ$ such that $(X,Δ)$ is klt.
\end{defn}

\begin{note}
  Recall that a pair $(X,Δ)$ as above is called klt\footnote{This is short for
    ``Kawamata log terminal''.} if the Weil $ℚ$-divisor $K_X+Δ$ is $ℚ$-Cartier,
  $ ⌊ Δ ⌋ = 0$, and if for one (equivalently: for every) resolution of
  singularities, $π : \wtilde{X} → X$, there exists a $ℚ$-linear equivalence of
  the form
  $$
  K_{\wtilde{X}} + π^{-1}_* Δ \sim_ℚ π^*(K_X+Δ)+ \textstyle{\sum_i} a_i·E_i,
  $$
  where the $E_i ⊂ \wtilde{X}$ are $π$-exceptional and the numbers $a_i ∈ ℚ$
  satisfy $a_i > -1$ for all $i$.
\end{note}

The results of Section~\ref{sec:03} rely in part on the Basepoint-free Theorem
and its immediate Corollary~\ref{cor:descprep}, which proves
Theorem~\ref{thm:desc1} for line bundles.  We recall the statement for the
reader's convenience.  Full details are found in the standard references, cf.\
including \cite[Thm.~3-1-1 and Rem.~3-1-2(i)]{KMM87}, \cite[Thm.~3.24]{KM98} and
\cite[Thm.~2.1.27]{Laz04-I}.

\begin{thm}[Basepoint-free Theorem]\label{thm:bpf}
  Let $φ : X → Y$ be a projective surjection of normal, quasi-projective
  varieties.  Assume that there exists an effective Weil $ℚ$-divisor $Δ$ on $X$
  such that $(X,Δ)$ is klt.  If $L$ is any $φ$-nef Cartier divisor on $X$ and
  $m ∈ ℕ^+$ any number such that $m·L - (K_X+Δ)$ is $φ$-nef and $φ$-big, then
  there exists a unique factorisation via a normal variety $Z$,
  \begin{equation}\label{eq:bpf}
    \begin{gathered}
      \xymatrix{ %
        X \ar@{->>}[r]_{α} \ar@{->>}@/^3mm/[rr]^{φ} & Z \ar@{->>}[r]_{β} & Y, %
      }
    \end{gathered}
  \end{equation}
  such that the following holds.
  \begin{enumerate}
  \item\label{il:t1} The morphisms $α$ and $β$ are surjective.  The morphism $α$
    has connected fibres.
  \item\label{il:t2} The Cartier divisor $L$ is the pull-back of a $β$-ample
    Cartier divisor $L_Z$ on $Z$.
  \item\label{il:t3} If $Y° ⊆ Y$ is open with preimages $X°$ and $Z°$, and if
    $L|_{X°}$ is $φ|_{X°}$-ample, then $α|_{X°} : X° → Z°$ is isomorphic.  \qed
  \end{enumerate}
\end{thm}

\begin{note}
  Item~\ref{il:t2} implies that the morphism $α $ contracts exactly those
  irreducible curves $C$ with $[L]·C = 0$ and $\dim \varphi(C) = 0$.
\end{note}

\begin{cor}[Descent of invertible sheaves]\label{cor:descprep}
  Let $φ : X → Y$ be a birational, projective morphism of normal,
  quasi-projective varieties, and let $L$ be any $φ$-numerically trivial Cartier
  divisor on $X$.  If there exists a Weil $ℚ$-divisor $Δ$ on $X$ such that
  $(X,Δ)$ is klt and such that $-(K_X+Δ)$ is $φ$-nef, then $L$ is linearly
  equivalent to the pull-back of a Cartier divisor on $Y$.
\end{cor}
\begin{proof}
  We claim that $L$ is $φ$-nef and that $D := L - \bigl( K_{X} + Δ \bigr)$ is
  $φ$-nef and $φ$-big.  Relative nefness of $L$ and $D$ holds by assumption.
  The condition that $D$ is $φ$-big is void since $φ$ is assumed to be
  birational.  Theorem~\ref{thm:bpf} thus gives a factorisation of $φ$ via a
  normal variety $Z$ as in \eqref{eq:bpf}, and $β$-ample Cartier divisor $L_Z$
  on $Z$ such that $L \sim α^*L_Z$.

  As a next step, we claim that $β$ is finite.  If not, we would find a curve
  $C ⊆ X$ which is mapped to a point by $φ$, but not by $α$.  The image curve
  $α(C)$ would them be contained in a $β$-fibre, and would thus have positive
  intersection with the $β$-ample divisor $L_Z$.  In particular,
  $\deg L|_C = \deg α^*(L_Z)|_C > 0$, contradicting the assumption that $L$ is
  $φ$-numerically trivial.

  In summary, see that $β$ is both birational and finite.  Zariski's main
  theorem, \cite[Cor.~12.88]{MR2675155}, applies to show that $β$ is isomorphic.
  Corollary~\ref{cor:descprep} follows.
\end{proof}

We list some basic properties of the contraction morphism associated with an
extremal face and refer to \cite[Def.~3-2-3]{KMM87} for terminology.

\begin{prop}[Relative contractions of extremal faces]\label{prop:rcoef}
  Assume we are given a sequence of projective surjections between normal,
  quasi-projective varieties, $f: X → Y$ and $Y → Z$.  Assume also that there
  exists an effective Weil $ℚ$-divisor $Δ$ on $X$ such that $(X,Δ)$ is klt.
  \begin{enumerate}
  \item\label{il:1w} If $f$ has only connected fibres and if $-(K_X+Δ)$ is
    $f$-ample, then there exists a $(K_X+Δ)$-negative extremal face $F$ of
    $\overline{NE}(X/Z)$ such that $f$ is the contraction morphism of $F$.
  \item\label{il:2w} If there exists a $(K_X+Δ)$-negative extremal face $F$ of
    $\overline{NE}(X/Z)$ such that $f$ is the contraction morphism of $F$, then
    relative Picard numbers are additive,
    $$
    ρ(X/Z) = ρ(X/Y)+ρ(Y/Z).
    $$
    The extremal face $F$ is an extremal ray if and only if $ρ(X/Y)=1$.
  \end{enumerate}
\end{prop}
\begin{proof}
  Item~\ref{il:1w} is \cite[Lem.~3-2-5(1)]{KMM87}.  The additivity of Picard
  numbers in Item~\ref{il:2w} is \cite[Lem.~3-2-5(3)]{KMM87}.  It remains to
  consider the relation between Picard numbers and the dimension of $F$ in
  Item~\ref{il:2w}.
  
  If the dimension of $F$ is one, it follows from the definition that
  $ρ(X/Y) = 1$ once we know that there \emph{is} a curve in $X$ that is mapped
  to a point by $f$.  This is shown in \cite[Lem.~3-2-4]{KMM87}.  As for the
  converse direction, if $ρ(X/Y) = 1$, the definition implies that there must be
  curves that are contracted.  The face $F$ cannot be empty, and is necessarily
  of dimension one.
\end{proof}

\begin{warning}
  Picard numbers are generally not additive for compositions of arbitrary
  morphisms.  See \cite[Sect.~2.2]{KM98} for a discussion and
  \cite[Rem.~3-2-6]{KMM87} for an explicit example.
\end{warning}

\subsection{Sheaves with operators and Higgs sheaves}
\label{ssec:Higgs}
\approvals{Behrouz & yes \\ Daniel & yes \\ Stefan & yes \\ Thomas & yes}

Higgs sheaves and sheaves with operators on singular varieties were defined and
discussed in detail in \cite[Sects.~4 and 5]{GKPT15}.  We briefly recall the
main definitions here and then discuss pullback of Higgs sheaves.

\subsubsection{Fundamental definitions}
\label{subsubsect:funddefs}

---

\begin{defn}[\protect{Sheaf with an operator, invariant subsheaves, \cite[Def.~4.1]{GKPT15}}]\label{def:nshfop1}
  Let $X$ be a normal, quasi-projective variety and $𝒲$ be a coherent sheaf of
  $𝒪_X$-modules.  A \emph{sheaf with a $𝒲$-valued operator} is a pair $(ℰ, θ)$
  where $ℰ$ is a coherent sheaf and $θ: ℰ → ℰ ⊗ 𝒲$ is an $𝒪_X$-linear sheaf
  morphism.
\end{defn}

\begin{defn}[\protect{Invariant subsheaf, \cite[Def.~4.8]{GKPT15}}]\label{def:ninvarSS}
  Setting as in Definition~\ref{def:nshfop1}.  A coherent subsheaf $ℱ ⊆ ℰ$ is
  called \emph{$θ$-invariant} if $θ(ℱ)$ is contained in the image of the natural
  map $ℱ ⊗ 𝒲 → ℰ ⊗ 𝒲$.  Call $ℱ$ \emph{generically invariant} if the restriction
  $ℱ|_U$ is invariant with respect to $θ|_U$, where $U ⊆ X$ is the maximal,
  dense, open subset where $𝒲$ is locally free.
\end{defn}

\begin{defn}[\protect{Stability of sheaves with operator, \cite[Def.~4.13]{GKPT15}}]\label{defn:swostab1}
  Let $X$ be a normal, projective variety and $H$ be any nef, $ℚ$-Cartier
  $ℚ$-divisor on $X$.  Let $(ℰ, θ)$ be a sheaf with an operator, as in
  Definition~\ref{def:nshfop1}, were $ℰ$ is torsion free.  We say that $(ℰ, θ)$
  is \emph{semistable with respect to $H$} if the inequality $μ_H(ℱ) ≤ μ_H(ℰ)$
  holds for all generically $θ$-invariant subsheaves $ℱ ⊆ ℰ$ with
  $0 < \rank ℱ < \rank ℰ$.  The pair $(ℰ, θ)$ is called \emph{stable with
    respect to $H$} if strict inequality holds.  Direct sums of stable sheaves
  with operator are called \emph{polystable}.
\end{defn}

On a singular variety, some attention has to be paid concerning the definition
of ``Higgs sheaf'' at singular points.  Again, we recall the relevant
definitions here.

\begin{defn}[\protect{Higgs sheaf, stability, morphisms, \cite[Defs.~5.1 and 5.2, Sect.~5.6]{GKPT15}}]\label{def:Higgs}
  Let $X$ be a normal variety.  A \emph{Higgs sheaf} is a pair $(ℰ, θ)$ of a
  coherent sheaf $ℰ$ of $𝒪_X$-modules, together with an $Ω^{[1]}_X$-valued
  operator $θ : ℰ → ℰ ⊗ Ω^{[1]}_X$, called \emph{Higgs field}, such that the
  composed morphism
  $$
  \xymatrix{ %
    ℰ \ar[rr]^(.35){θ} && ℰ ⊗ Ω^{[1]}_X \ar[rr]^(.45){θ ⊗ \Id} && ℰ ⊗ Ω^{[1]}_X
    ⊗ Ω^{[1]}_X \ar[rr]^(.55){\Id ⊗ [Λ]} && ℰ ⊗ Ω^{[2]}_X }
  $$
  vanishes.  A Higgs sheaf is called stable if it is stable as a sheaf with an
  $Ω^{[1]}_X$-valued operator, ditto for semistable and polystable.  A
  \emph{morphism of Higgs sheaves}, written $f : (ℰ_1, θ_1) → (ℰ_2, θ_2)$, is a
  morphism $f : ℰ_1 → ℰ_2$ of coherent sheaves that commutes with the Higgs
  fields, $(f ⊗ \Id_{Ω^{[1]}_X}) ◦ θ_1 = θ_2 ◦ f$.
\end{defn}

\subsubsection{Pull-back}
\label{sssec:pb}
\approvals{Behrouz & yes \\ Daniel & yes \\ Stefan & yes \\ Thomas & yes}

If $f : Y → X$ is a morphism of normal varieties, and if $(ℰ,θ)$ is a Higgs
sheaf on $X$, there is generally no way to equip the pull-back sheaf $f^*ℰ$ with
a Higgs field, even if $ℰ$ is locally free.  It is a non-trivial fact that
pull-backs do exist for klt spaces.  We refer to \cite[Sects.~5.3 and
5.4]{GKPT15} for details.  In brief, if $X$ is klt, recall from \cite[Thms.~1.3
and 5.2]{MR3084424} that there exists a natural pull-back functor for reflexive
differentials on klt pairs that is compatible with the usual pull-back of Kähler
differentials and gives rise to a sheaf morphism
$$
d_{\refl} f : f^* Ω^{[1]}_X → Ω^{[1]}_Y.
$$
One can then define a Higgs field on $f^*ℰ$ as the composition of the following
morphisms,
\begin{equation}\label{eq:pbhgs}
  f^*ℰ \xrightarrow{f^* θ} f^* \Bigl( ℰ ⊗ Ω^{[1]}_X \Bigr) = f^*
  ℰ ⊗ f^* Ω^{[1]}_X \xrightarrow{\Id_{f^* ℰ} ⊗
    d_{\refl} f} f^* ℰ ⊗ Ω^{[1]}_Y.
\end{equation}

\subsection{Stability for sheaves on the smooth locus}
\approvals{Behrouz & yes \\ Daniel & yes \\ Stefan & yes \\ Thomas & yes}
\label{ssec:stability}

In the situation discussed in the present paper, it makes sense to generalise
the stability notions of Section~\ref{ssec:Higgs} to the case where the (Higgs-)
sheaves are defined on the smooth locus of a normal variety only.

\begin{defn}[Slope and stability for sheaves on the smooth locus]\label{def:sHsl1}
  Let $X$ be a normal, projective variety and let $ℰ°$ be a torsion free,
  coherent sheaf on $X_{\reg}$ of positive rank.  If $H ∈ \Div(X)$ is nef,
  define the \emph{slope of $ℰ°$ with respect to $H$} as
  $$
  μ_H(ℰ°) := \frac{c_1(ι_* ℰ°)·[H]^{\dim X-1}}{\rank ℰ°},
  $$
  where $ι : X_{\reg} → X$ is the inclusion.
\end{defn}

\begin{rem}[Algebraicity assumption]
  We underline that $ℰ°$ is assumed to be algebraic in
  Definition~\ref{def:sHsl1}.  For coherent \emph{analytic} sheaves on
  $X_{\reg}^{an}$, the push-forward $ι_* ℰ°$ need not be coherent in general.
\end{rem}

\begin{defn}[Slope and stability for sheaves on the smooth locus]\label{def:sHsl2}
  Setting as in Definition~\ref{def:sHsl1}.  If $𝒲°$ is coherent on $X_{\reg}$
  and if $θ° : ℰ° → ℰ° ⊗ 𝒲°$ is a $𝒲°$-valued operator, say that
  \emph{$(ℰ°, θ°)$ is stable with respect to $H$}, if the inequality
  $μ_H(ℱ°) < μ_H(ℰ°)$ holds for all generically $θ°$-invariant subsheaves
  $ℱ° ⊆ ℰ°$ with $0 < \rank ℱ° < \rank ℰ°$.  Analogously, define notions of
  semistable and polystable for sheaves with operators on $X_{\reg}$, ditto for
  Higgs sheaves.
\end{defn}

The following two lemmas summarise properties of the generalised stability
notions that will be used later.  Proofs are elementary and therefore omitted.

\begin{lem}[Restriction of stable sheaves to $X_{\reg}$]\label{lem:owe1}
  Let $X$ be a normal, projective variety, and let $(ℰ, θ)$ be a torsion free
  sheaf with a $𝒲$-valued operator.  If $H ∈ \Div(X)$ is nef, then $(ℰ, θ)$ is
  semistable (resp.\ stable) with respect to $H$ as a sheaf with a $𝒲$-valued
  operator if and only if $(ℰ, θ)|_{X_{\reg}}$ is semistable (resp.\ stable)
  with respect to $H$ as a sheaf with a $𝒲|_{X_{\reg}}$-valued operator.  \qed
\end{lem}

\begin{lem}[Extensions of operators from subsheaves]\label{lem:owe2}
  Let $X$ be a normal, projective variety and let $(ℰ°, θ°)$ be a torsion free
  sheaf on $X_{\reg}$ with a $𝒲°$-valued operator.  Let $ι° : 𝒲° ↪ \sV°$ is
  an inclusion of coherent sheaves on $X_{\reg}$ and consider the natural
  $\sV°$-valued operator $τ°$ that is defined as the composition
  $$
  \xymatrix{ %
    ℰ° \ar[rr]_(.4){θ°}\ar@/^0.4cm/[rrrr]^{τ°} && ℰ° ⊗ 𝒲° \ar[rr]_{\Id ⊗ ι°} &&
    ℰ° ⊗ \sV°.  }
  $$
  If $H ∈ \Div(X)$ is nef, then, $(ℰ°, θ°)$ semistable (resp.\ stable) with
  respect to $H$ as a sheaf with a $𝒲°$-valued operator if and only if
  $(ℰ°, τ°)$ semistable (resp.\ stable) with respect to $H$ as a sheaf with a
  $\sV°$-valued operator.  \qed
\end{lem}

%
% Do not edit the following line.  The text is automatically updated by
% subversion.
%
\svnid{$Id: 03.tex 1180 2018-10-25 09:04:48Z kebekus $}

\section{The nonabelian Hodge correspondence for locally free Higgs sheaves}\label{sect:NAHT}
\subversionInfo

\subsection{Main result}
\label{ssect:MainResult}
\approvals{Behrouz & yes \\ Daniel & yes \\ Stefan & yes \\ Thomas & yes}

Generalising work of Simpson, \cite[Cor.~3.10]{MR1179076}, in this section we
formulate a nonabelian Hodge correspondence for locally free Higgs sheaves on
klt spaces, relating such locally free Higgs sheaves to representations of the
fundamental group.  The formulation uses the fact that every algebraic variety
admits a distinguished, \emph{canonical} resolution of singularities, see
\cite[Rem.~1.16ff and Sect.~13]{BM96}.  The following additional notation will
be used.

\begin{notation}[Categories in the nonabelian Hodge correspondence]\label{not:category}
  Given a normal, projective variety $X$, consider the following categories.
  \begin{enumerate}
  \item[$\Higgs_X$]\label{il:x1} Locally free Higgs sheaves $(ℰ,θ)$ on $X$
    having the property that there exists an ample divisor $H ∈ \Div(X)$ such
    that $(ℰ, θ)$ is semistable with respect to $H$ and additionally satisfies
    $ch_1(ℰ)·[H]^{n-1} = ch_2(ℰ)·[H]^{n-2} = 0$.

  \item[$\LSys_x$] Local systems on $X$.
  \end{enumerate}
\end{notation}

\begin{notation}[Higgs sheaves]
  There is no uniform definition of Higgs sheaves on singular spaces in the
  literature.  Throughout the present paper, we use the definition given in
  Section~\ref{subsubsect:funddefs}.  (Semi)stability of Higgs sheaves is
  defined and discussed in \cite[Sect.~5.6]{GKPT15}.  We refer to
  \cite[Sect.~I.1]{Deligne70} for a discussion of the basic properties of local
  systems.
\end{notation}

\begin{notation}[Nonabelian Hodge correspondence for manifolds]
  If $X$ is smooth, Simpson's nonabelian Hodge correspondence gives an
  equivalence between the categories $\Higgs_X$ and $\LSys_X$.  We denote the
  relevant functors by $η_X : \LSys_X → \Higgs_X$ and
  $μ_X : \Higgs_X → \LSys_X$.
\end{notation}

The following is now the main result.  As we will see in
Section~\ref{ssect:MainResultFunc} below, the properties spelled out
in \ref{il:HC2} and \ref{il:HC3} immediately imply that the nonabelian Hodge
correspondence presented here is in fact fully functorial in morphisms of klt
spaces.

\begin{thm}[Nonabelian Hodge correspondence for klt spaces]\label{thm:free-naht}
  For every projective klt space $X$, there exists an equivalence of categories,
  given by functors $η_X : \LSys_X → \Higgs_X$ and $μ_X : \Higgs_X → \LSys_X$
  such that the following additional properties hold.
  \begin{enumerate}
  \item\label{il:HC1} If $X$ is smooth, then $η_X$ and $μ_X$ equal the functors
    from Simpson's nonabelian Hodge correspondence.

  \item\label{il:HC2} If $(ℰ,θ) ∈ \Higgs_X$ and $π: \wtilde{X} → X$ is the
    canonical resolution of singularities, then
    $π^*(ℰ,θ) ∈ \Higgs_{\wtilde{X}}$, and there exists a canonical isomorphism
    of local systems,
    $$
    M_{π,(ℰ,θ)} : π^* μ_X(ℰ,θ) → μ_{\wtilde{X}}\bigl( π^*(ℰ,θ) \bigr).
    $$

  \item\label{il:HC3} If $\aE ∈ \LSys_X$ is any local system and
    $π: \wtilde{X} → X$ is the canonical resolution of singularities, then there
    exists a canonical isomorphism of Higgs sheaves,
    $$
    N_{π,\aE} : π^* η_X(\aE) → η_{\wtilde{X}}\bigl( π^* \aE \bigr).
    $$
  \end{enumerate}
\end{thm}

\begin{note}[Pull-back of Higgs sheaves]
  Item~\ref{il:HC2} discusses the pull-back of the Higgs sheaf $η_X(\aE)$ from
  $X$ to the resolution of singularities, $\wtilde{X}$, as discussed in
  Section~\ref{sssec:pb} above.
\end{note}

Theorem~\ref{thm:free-naht} is shown in Section~\ref{sec:pfnh2}.
Section~\ref{ssec:ascend} prepares for the proof.

\subsection{Functoriality}\label{ssect:MainResultFunc}
\approvals{Behrouz & yes \\ Daniel & yes \\ Stefan & yes \\ Thomas & yes}

Items~\ref{il:HC2} and \ref{il:HC3} of Theorem~\ref{thm:free-naht} allow to
describe the functors $η_X$ and $μ_X$ on any given klt space $X$ in terms of the
classical nonabelian Hodge correspondence that exists on the canonical
resolution of singularities.  In fact, a much more general functoriality holds
true.

\begin{thm}[Functoriality in morphisms]\label{thm:finm}
  The correspondence of Theorem~\ref{thm:free-naht} is functorial in morphisms.
  More precisely, for every morphism $f: Y → X$ of projective klt spaces, every
  $\aE ∈ \LSys_X$ and every $(ℰ,θ) ∈ \Higgs_X$, there exist canonical
  isomorphisms
  $$
  M_{f,(ℰ,θ)} : f^* μ_X(ℰ,θ) → μ_Y\bigl( f^*(ℰ,θ) \bigr)
  \quad\text{and}\quad N_{f,\aE} : f^* η_X(\aE) → η_Y\bigl( f^* \aE \bigr).
  $$
  The collection of these isomorphisms satisfies the following properties.
  \begin{description}
  \item[Functoriality] Given morphisms $g: Z → Y$ and $f: Y → X$ between
    projective klt spaces and $\aE ∈ \LSys_X$, then the following diagram
    commutes,
    $$
    \xymatrix{ %
      g^* f^* η_X(\aE) \ar[rr]_{g^* N_{f,\aE}} \ar@/^.4cm/[rrrr]^{N_{f◦ g, \aE}} && g^* η_Y \bigl(f^* \aE \bigr) \ar[rr]_{N_{g,f^*\aE}} && η_Z \bigl(g^* f^* \aE \bigr).
    }
    $$
    Ditto for the functor $μ_•$ and the isomorphisms $M_{•,•}$.

  \item[Behaviour under canonical resolution] For $π : \wtilde{X} → X$ the
    canonical resolution of a projective klt space, the morphisms $M_{π,•}$ and
    $N_{π,•}$ equal the isomorphisms given in Items~\ref{il:HC2} and
    \ref{il:HC3} of Theorem~\ref{thm:free-naht}.

  \item[Compatibility] If $f : Y → X$ is a morphism between smooth projective
    varieties, then $M_{f,•}$ and $N_{f,•}$ are the standard isomorphisms given
    by functoriality of Simpson's nonabelian Hodge correspondence.
  \end{description}
\end{thm}

Theorem~\ref{thm:finm} is shown in Section~\ref{ssec:gghfj1} below.

\begin{rem}[Uniqueness]
  One verifies in the blink of an eye that functoriality, behaviour under
  canonical resolution and compatibility determine the isomorphisms $M_{•, •}$
  and $N_{•, •}$ uniquely.
\end{rem}

\begin{rem}[General resolutions]
  Theorem~\ref{thm:finm} implies in particular that the statement of
  Theorem~\ref{thm:free-naht} holds for any resolution of singularities, not
  just the canonical resolution.  Taken together with \cite{Takayama2003}, for
  any klt space this establishes an equivalence of categories of $\Higgs_X$ with
  $\Higgs_{\wtilde{X}}$, where $\wtilde{X}$ is any resolution of singularities
  of $X$.
\end{rem}

As a further consequence of functoriality, we observe that the nonabelian Hodge
correspondence respects group actions and relates $G$-linearised local systems
to Higgs $G$-sheaves in the sense of \cite[Def.~5.1]{GKPT15}.

\begin{cor}[$G$-linearised local systems and Higgs $G$-sheaves]\label{cor:gllshs}
  Let $X$ be a projective klt space, and let $G$ be a group acting on $X$ via a
  group morphism $G → \Aut(X)$.
  \begin{enumerate}
  \item If $(ℰ, θ) ∈ \Higgs_X$ carries the structure of a Higgs $G$-sheaf, given
    by isomorphisms $φ_g : g^* (ℰ, θ) → (ℰ, θ)$, then the following composed
    maps endow the local system $μ_X(ℰ, θ)$ with a $G$-linearisation,
    $$
    \xymatrix{ %
      g^* μ_X(ℰ,θ) \ar[rr]^{M_{g,(ℰ,θ)}} && μ_X\bigl(g^*(ℰ,θ)\bigr) \ar[rr]^{μ_X(φ_g)} && μ_X(ℰ,θ).
    }
    $$

  \item If $\aE ∈ \LSys_X$ carries a $G$-linearisation given by isomorphisms
    $φ_g : g^* \aE → \aE$, then the following composed maps endow $η_X(\aE)$
    with the structure of a Higgs $G$-sheaf,
    $$
    \xymatrix{ %
      g^* η_X(\aE) \ar[rr]^{N_{g,\aE}} && η_X \bigl(g^*\aE\bigr) \ar[rr]^{η_X(φ_g)} && η_X(\aE).
    } \eqno \qed
    $$
  \end{enumerate}
\end{cor}

\subsection{Independence of polarisation}
\approvals{Behrouz & yes \\ Daniel & yes \\ Stefan & yes \\ Thomas & yes}

As in Simpson's original setup, a Higgs bundle on a klt space $X$ is in
$\Higgs_X$ if and only if it satisfies the conditions of Notation~\ref{il:x1}
with respect to any ample class.  The following proposition makes this assertion
precise.

\begin{thm}[Independence of polarisation]\label{thm:iop}
  Let $X$ be a projective klt space of dimension $n$.  Given any locally free
  Higgs sheaf $(ℰ,θ)$ on $X$, the following statements are equivalent.
  \begin{enumerate}
  \item\label{il:q1} There exists an ample divisor $H ∈ \Div(X)$ such that
    $ch_1(ℰ)·[H]^{n-1} = ch_2(ℰ)·[H]^{n-2} = 0$ and such that $(ℰ,θ)$ is
    semistable with respect to $H$.

  \item\label{il:q2} For all ample divisors $H ∈ \Div(X)$, we have
    $ch_1(ℰ)·[H]^{n-1} = ch_2(ℰ)·[H]^{n-2} = 0$ and $(ℰ,θ)$ is semistable with
    respect to $H$.

  \item\label{il:q3} All Chern classes $c_i(ℰ) ∈ H^{2i}\bigl( X,\, ℚ \bigr)$
    vanish and $(ℰ,θ)$ is semistable with respect to any ample divisor on $X$.

  \item\label{il:q4} There exists a resolution of singularities,
    $π: \wtilde{X} → X$, and an ample divisor $\wtilde{H} ∈ \Div(\wtilde{X})$
    such that
    $ch_1(π^*\,ℰ)·[\wtilde{H}]^{n-1} = ch_2(π^*\,ℰ)·[\wtilde{H}]^{n-2} = 0$ and
    such that $π^*\,(ℰ,θ)$ is semistable with respect to $\wtilde{H}$.

  \item\label{il:q5} For any resolution of singularities, $π: \wtilde{X} → X$,
    and any ample divisor $\wtilde{H} ∈ \Div(\wtilde{X})$, we have intersection
    numbers
    $ch_1(π^*\,ℰ)·[\wtilde{H}]^{n-1} = ch_2(π^*\,ℰ)·[\wtilde{H}]^{n-2} = 0$ and
    $π^*\,(ℰ,θ)$ is semistable with respect to $\wtilde{H}$.
  \end{enumerate}
  The analogous equivalences hold when ``semistable'' is replaced by ``stable''
  or ``polystable''.
\end{thm}

Theorem~\ref{thm:iop} is shown in Section~\ref{ssec:gghfj2} below.

\subsection{Harmonic bundles, differential graded categories}
\approvals{Behrouz & yes \\ Daniel & yes \\ Stefan & yes \\ Thomas & yes}

Simpson constructs his nonabelian Hodge correspondence first in the case of
polystable Higgs bundles and semisimple local systems.  In this setup, the
correspondence is a consequence of existence theorems for pluri-harmonic metrics
on the underlying bundles.  Given their central role in the theory, we remark
that the nonabelian Hodge correspondence for klt spaces,
Theorem~\ref{thm:free-naht} also has a description in terms of harmonic
structures, although in our case the harmonic metric exists on the smooth locus
of the underlying space only.  The proof of the following proposition is simple
and therefore omitted: Item~\ref{il:HC3} of Theorem~\ref{thm:free-naht} allows
to relate the correspondence on $X$ to that on a resolution.

\begin{prop}[Hodge correspondence for klt spaces via harmonic bundles]\label{ex:naht2a}
  Let $X$ be a projective, klt space and $\aE ∈ \LSys_X$ a semisimple local
  system on $X$, with underlying $\cC^{∞}$-bundle $E$.  We claim that there
  exists a tame and purely imaginary harmonic bundle
  $(E°, \bar{∂}_{E°}, θ°, h°)$ on $X_{\reg}$ with the following two properties.
  \begin{enumerate}
  \item\label{il:by1} The induced flat bundle $(E°,∇_{h°})$ corresponds to the
    local system $\aE|_{X_{\reg}}$.

  \item\label{il:by2} Writing $ℰ°$ for the sheaf of holomorphic sections in
    $(E°, \bar{∂}_{E°})$, the functor $η_X$ of the nonabelian Hodge
    correspondence satisfies $(ℰ°,θ°) ≅ η_X(\aE)|_{X_{\reg}}$.  \qed
  \end{enumerate}
\end{prop}

\begin{rem}
  We refer to \cite[Sect.~1]{MR3087348} for a discussion of harmonic bundles,
  and to \cite[p.~723]{MR1040197} and \cite[Sect.~22.1]{MR2283665} for the
  notions of ``tame'' and ``purely imaginary''.
\end{rem}

As a second point, we note that Theorem~\ref{thm:free-naht} includes
equivalences of differential graded categories (=DGCs) that appear in Simpson's
nonabelian Hodge theory.  For detailed discussion of DGCs and their relation to
this theory we refer the reader to~\cite[Sect.~3]{MR1179076} and the references
therein.  Given a resolution $π: \wtilde{X} → X$, Theorem~\ref{thm:free-naht}
implies that there is an equivalence of DGCs between the category of extensions
of stable Higgs bundles with vanishing Chern classes on $X$ and the category of
flat connections on $\wtilde X$.  The same conclusions as in the smooth
projective setting then follow.

%
% Do not edit the following line.  The text is automatically updated by
% subversion.
%
\svnid{$Id: 04.tex 1175 2018-08-06 07:05:24Z taji $}

\section{Descent of vector bundles from klt spaces}\label{sec:03}
\subversionInfo
\approvals{Behrouz & yes \\ Daniel & yes \\ Stefan & yes \\ Thomas & yes}

The proof of Theorem~\ref{thm:desc2}, our main result concerning descent of
vector bundles to klt spaces, relies on the following auxiliary statement, which
we prove in this section.

\begin{thm}[Descent of vector bundles from klt spaces]\label{thm:desc1}
  Let $φ : X → Y$ be a projective, birational morphism of normal,
  quasi-projective varieties.  Assume that there exists a Weil $ℚ$-divisor $Δ_X$
  on $X$ such that the pair $(X,Δ_X)$ is klt and $-(K_X+Δ_X)$ is $φ$-nef.  If
  $ℱ_X$ is any locally free, $φ$-numerically flat sheaf on $X$, then there
  exists a locally free sheaf $ℱ_Y$ on $Y$ such that $ℱ_X ≅ φ^* ℱ_Y$.
\end{thm}

\subsection{Proof of Theorem~\ref{thm:desc1}: Setup and notation}
\approvals{Behrouz & yes \\ Daniel & yes \\ Stefan & yes \\ Thomas & yes}

We maintain notation and assumptions of Theorem~\ref{thm:desc1} throughout the
present Section~\ref{sec:03}.  To avoid trivial cases, we may assume throughout
the proof that the rank of $ℱ_X$ is positive.  Set $r := (\rank ℱ_X)-1$.  Using
Grothendieck's terminology, we consider the associated $ℙ^r$-bundle
$ℙ_X := ℙ_X(ℱ_X)$.  The following diagram summarises the situation
\begin{equation}\label{eq:gkbn1}
  \begin{gathered}
    \xymatrix{ %
      ℙ_X \ar[d]_{ρ_X} \ar@/^2mm/[rrd]^{δ} \\
      X \ar[rr]_{φ\text{, birational}} && Y.
    }
  \end{gathered}
\end{equation}
Since $ρ_X$ is a locally trivial $ℙ^r$-bundle, the variety $ℙ_X$ is normal, and
the pair $(ℙ_X, Δ_{ℙ_X})$ is klt, where $Δ_{ℙ_X} := ρ_X^* Δ_X$.  Let $Y° ⊆ Y$ be
the maximal open set over which $φ$ is isomorphic, and write $X° := φ^{-1}(X°)$.
As $Y$ is normal, the subset $Y°$ is big.

We will also consider the invertible sheaves $ℒ_X := 𝒪_{ℙ_X(ℱ_X)}(1)$ and
$\sM_X := \det ℱ_X$.  The assumption that $ℱ_X$ is $φ$-numerically flat has
immediate consequences for $ℒ_X$ and $\sM_X$, which we state for later
reference.

\begin{obs}
  ---
  \begin{enumerate}
  \item\label{obs:1} The sheaf $ℒ_X$ is $δ$-nef.
  \item\label{obs:2} The sheaf $\sM_X$ is $φ$-numerically trivial.  In
    particular, Corollary~\ref{cor:descprep} implies the existence of an
    invertible sheaf $\sM_Y ∈ \Pic(Y)$ such that $\sM_X ≅ φ^* \sM_Y$.  \qed
  \end{enumerate}
\end{obs}

For convenience of notation, choose Cartier divisors $L_X$ and $M_Y$
representing the bundles $ℒ_X$ and $\sM_Y$.  The Cartier divisor
$M_X := φ^* M_Y$ will then represent $\sM_X$.

\subsection{Proof of Theorem~\ref{thm:desc1}: Factorisation of $δ$}
\approvals{Behrouz & yes \\ Daniel & yes \\ Stefan & yes \\ Thomas & yes}

We aim to construct a locally free sheaf $ℱ_Y$ on $Y$.  Rather than doing so
directly, we will first construct a factorisation of $δ$ via a morphism
$ρ_Y : ℙ_Y → Y$ that agrees with $ℙ_X$ over the big open set where $Y°$.  Later,
we will show that $ρ_Y$ is equidimensional, and has in fact the structure of a
linear $ℙ^r$-bundle.  The sheaf $ℱ_Y$ will then be constructed as the
push-forward of the relative hyperplane bundle on $ℙ_Y$.

\begin{claim}[Construction of $ℙ_Y$]\label{claim:desc1}
  There exists a commutative diagram of surjective projective morphisms with
  connected fibres extending Diagram~\eqref{eq:gkbn1} as follows,
  \begin{equation}\label{eq:bcd2}
    \begin{gathered}
      \xymatrix{ %
        ℙ_X \ar[rr]^{Φ} \ar[d]_{ρ_X} \ar[drr]_{δ} && ℙ_Y \ar[d]^{ρ_Y} \\
        X \ar[rr]_{φ\text{, birational}} && Y, %
      }
    \end{gathered}
  \end{equation}
  as well as a $ρ_Y$-ample sheaf $ℒ_Y ∈ \Pic(ℙ_Y)$ such that $ℒ_X ≅ Φ^* ℒ_Y$.
  The variety $ℙ_Y$ and the morphisms of Diagram~\eqref{eq:bcd2} are unique up
  to isomorphism, and the following holds in addition.
  \begin{enumerate}
  \item\label{il:C} The restricted morphism
    $Φ|_{ρ_X^{-1}(X°)}: ρ_X^{-1}(X°) → ρ_Y^{-1}(Y°)$ is isomorphic.  In
    particular, $Φ$ is birational.
  \item\label{il:remus} We have an isomorphism of sheaves
    $\sM_Y ≅ \det \bigl( (ρ_Y)_*\, ℒ_Y \bigr)$.
  \end{enumerate}
\end{claim}
\begin{proof}[Proof of Claim~\ref*{claim:desc1}]
  The factorisation of $δ$ via an intermediate variety $ℙ_Y$ will be constructed
  using Theorem~\ref{thm:bpf} (``Basepoint-free Theorem'').  To apply the
  theorem, it suffices to show that $L_X$ is $δ$-nef and that
  $$
  D := L_X - \bigl( K_{ℙ_X} + Δ_{ℙ_X} \bigr)
  $$
  is $δ$-nef and $δ$-big.  Relative nefness of $L_X$ is clear from
  Observation~\ref{obs:1}.  To analyse $D$, we use the standard
  formula\Preprint{\footnote{$-K_{ℙ_X} \sim (r+1)·L_X - ρ_X^*(K_X + M_X)$}} for
  the canonical bundle of a projectivised vector bundle to obtain a $ℚ$-linear
  equivalence of $ℚ$-divisors,
  \begin{equation}\label{eq:sdsmdf}
    D \sim_{ℚ} (r+2)·L_X - ρ_X^*(K_X+Δ_X+M_X).
  \end{equation}
  The divisor $L_X$ is $ρ_X$-ample and therefore ample on the general fibre of
  $δ$.  Since $φ$ is birational, Equation~\eqref{eq:sdsmdf} implies that $D$ is
  $δ$-big.  Relative nefness of $D$ also follows from
  Equation~\eqref{eq:sdsmdf}, using Observations~\ref{obs:1} and \ref{obs:2}, as
  well as the assumption that $-(K_X+Δ_X)$ is $φ$-nef.  Theorem~\ref{thm:bpf}
  thus applies and yields a unique factorisation as in Diagram~\eqref{eq:bcd2},
  as well as a $ρ_Y$-ample sheaf $ℒ_Y$ such that $ℒ_X ≅ Φ^* ℒ_Y$.

  Item~\ref{il:C} follows from Item~\ref{il:t3} of Theorem~\ref{thm:bpf}, using
  the fact that $D$ is $ρ_X$-ample over $X°$, and therefore $δ$-ample over $Y°$.
  It remains to show that $\sM_Y ≅ \det \bigl( (ρ_Y)_*\, ℒ_Y \bigr)$.  By
  Item~\ref{il:C}, an isomorphism exists at least over the big open set $Y°$.
  But since both sides of the equation are reflexive, the isomorphism exists
  globally.  This ends the proof of Claim~\ref{claim:desc1}.
\end{proof}

\subsection{Proof of Theorem~\ref{thm:desc1}: Equidimensionality}
\approvals{Behrouz & yes \\ Daniel & yes \\ Stefan & yes \\ Thomas & yes }

As indicated above, we will now show that $ρ_Y$ is equidimensional.  This is the
point where the assumption that $ℱ_X$ is $φ$-numerically flat is used in a
crucial way.

\begin{claim}[Equidimensionality of $ρ_Y$]\label{claim:desc2}
  The morphism $ρ_Y$ is equidimensional, of relative dimension $r$.
\end{claim}
\begin{proof}[Proof of Claim~\ref{claim:desc2}]
  We argue by contradiction and assume that there exists a point $y ∈ Y$ whose
  fibre $ρ_Y^{-1}(y)$ contains a subvariety $Z$ of dimension $r+1$.  Recalling
  that $ℒ_Y$ is $ρ_Y$-ample, we have a positive intersection number
  \begin{equation}\label{eq:hjkk}
    [ℒ_Y]^{r+1}·[Z] > 0.
  \end{equation}
  We will see that this is absurd.  To this end, let $F ⊆ φ^{-1}(y)$ be any
  irreducible component.  Choose a desingularisation $π: \wtilde{F} → F$ and
  extend Diagram~\eqref{eq:bcd2} to the left by taking fibre products as
  follows,
  $$
  \xymatrix{ %
    ℙ_{\wtilde F} \ar@{->>}[r]^{Π} \ar[d]_{ρ_{\wtilde F}} & ℙ_F \ar@{^(->}[r] \ar[d]_{ρ_F} & ℙ_X \ar[d]_{ρ_X} \ar[r]^{Φ} & ℙ_Y \ar[d]_{ρ_Y} \\
    \wtilde F \ar@{->>}[r]_{π} & F \ar@{^(->}[r] & X \ar[r]_{φ} & Y %
  } \qquad %
  \xymatrix{ %
    ℙ_{\wtilde F} \ar@{->>}[r]^{Π} \ar[d]^{ρ_{\wtilde F}} & ℙ_F \ar@{->>}[r]^(.4){Φ|_{ℙ_F}} \ar[d]^{ρ_F} & ρ_Y^{-1}(y) \ar[d] \ar@{^(->}[r] & ℙ_Y \ar[d]^{ρ_Y} \\
    \wtilde F \ar@{->>}[r]_{π} & F \ar@{->>}[r]_{φ|_F} & \{y\} \ar@{^(->}[r] & Y %
  }
  $$
  Setting $ℱ_{\wtilde F} := π^* ℱ_X$ and $ℒ_{\wtilde F} := Π^* ℒ_X$, we
  obtain identifications
  $$
  ℙ_{\wtilde F} ≅ ℙ_{\wtilde F}\bigl(ℱ_{\wtilde F}\bigr) \quad \text{and}
  \quad ℒ_{\wtilde F} ≅ 𝒪_{ℙ_{\wtilde F}(ℱ_{\wtilde F})}(1).
  $$
  If $Z_{\wtilde F} ⊆ ℙ_{\wtilde F}$ is any $r+1$-dimensional subvariety that
  dominates $Z$, Inequality~\eqref{eq:hjkk} immediately implies that
  \begin{equation}\label{eq:hjkk1}
    [ℒ_{\wtilde F}]^{r+1}·[Z_{\wtilde F}] = [Π^* Φ^*\,ℒ_Y]^{r+1}·[Z_{\wtilde F}] > 0.
  \end{equation}
  On that other hand, as the pullback of a numerically flat bundle,
  $ℱ_{\wtilde F}$ is numerically flat.  Theorem~\ref{thm:numflat} hence implies
  that all Chern classes $c_i\bigl(ℱ_{\wtilde F}\bigr)$ of $ℱ_{\wtilde F}$
  vanish.  A standard Chern class computation on projectivised vector bundles,
  \cite[Rem.~3.2.4 on p.~55]{Fulton98}, thus gives
  \begin{equation}\label{eq:hjkk2}
    c_1 \bigl( ℒ_{\wtilde F} \bigr)^{r+1} = %
    - \sum_{i=1}^{r} \underbrace{c_i \bigl( ρ_{\wtilde F}^* \, ℱ_{\wtilde F} \bigr)}_{= 0}·c_1
    \bigl( ℒ_{\wtilde F} \bigr)^{r+1-i} = 0.
  \end{equation}
  Items~\eqref{eq:hjkk1} and \eqref{eq:hjkk2} are obviously in contradiction.
  The assumption that $ρ_Y^{-1}(y)$ contains an $(r+1)$-dimensional subvariety
  is thus absurd.  In summary, we obtain that $ρ_Y$ is equidimensional, thus
  finishing the proof of Claim~\ref{claim:desc2}.
\end{proof}

Building on work of Kollár and Höring-Novelli, it has been shown by Araujo and
Druel \cite[Prop.~4.10]{MR3273645} that equidimensionality and the existence of
the relatively ample sheaf $ℒ_Y$ whose restriction to general $ρ_Y$-fibres is
the hyperplane bundle implies that $ρ_Y$ has the structure of a linear bundle.
We briefly recall the argument.

\begin{claim}[Linear bundle structure of $ρ_Y$]
  The sheaf $ℱ_Y := (ρ_Y)_* ℒ_Y$ is locally free on $Y$.  The morphism
  $ρ_Y : ℙ_Y → Y$ can be identified as the projection $ℙ_Y(ℱ_Y) → Y$.  The
  invertible sheaf $ℒ_Y$ becomes $𝒪_{ℙ_Y(ℱ_Y)}(1)$ under this identification.
\end{claim}
\begin{proof}
  By~\ref{il:C}, we know that the general fibre $ℙ_{Y,y}$ of $ρ_Y$ is isomorphic
  to $ℙ^r$, with $ℒ_Y|_{ℙ_{Y,y}} ≅ 𝒪_{ℙ^r}(1)$.  Since $ℒ_Y$ is $ρ_Y$-ample,
  \cite[Prop.~3.1]{MR3030002} applies to guarantee that in fact all fibres of
  $ρ_Y$ are irreducible and generically reduced, and that the normalisation of
  any fibre is isomorphic to $ℙ^r$.  In particular, the Hilbert polynomial of
  the normalisation of the fibres is constant, and \cite[Thm.~12]{MR2838215}
  applies to show that the variety $ℙ_Y$ admits a simultaneous normalisation,
  which is a finite, birational morphism $η : \what{ℙ}_Y → ℙ_Y$ with the
  property that the morphism $ρ_Y ◦ η$ is flat and that all its fibres are
  normal.  But since $ℙ_Y$ already is normal, Zariski's main theorem,
  \cite[V~Thm.~5.2]{Ha77}, applies to show that $η$ is isomorphic, and all
  fibres of $ρ_Y$ are therefore smooth, and isomorphic to $ℙ^r$.  Grauert's
  theorem, \cite[III~Cor.~12.9]{Ha77} thus shows that $ℱ_Y$ is locally free, and
  that the natural morphism $(ρ_Y)^*ℱ_Y = (ρ_Y)^* (ρ_Y)_* ℒ_Y → ℒ_Y$ is
  surjective.  The universal property of projectivisation,
  \cite[II~Prop.~7.12]{Ha77}, thus gives a morphism $α : ℙ_Y → ℙ_Y(ℱ_Y)$ that
  identifies $ℒ_Y$ with the pull-back of $𝒪_{ℙ_Y(ℱ_Y)}(1)$.  But since the
  identification $ℒ_Y|_{(ρ_Y)^{-1}(y)} ≅ 𝒪_{ℙ^r}(1)$ now holds for every
  $y ∈ Y$, the morphism $α$ is clearly bijective, hence isomorphic by Zariski's
  main theorem.
\end{proof}

\subsection{Proof of Theorem~\ref{thm:desc1}: End of proof}
\approvals{Behrouz & yes \\ Daniel & yes \\ Stefan & yes \\ Thomas & yes}

It remains to show that $ℱ_X ≅ φ^* ℱ_Y$.  Recalling that
$$
ℱ_Y = (ρ_Y)_* ℒ_Y \quad\text{and}\quad ℱ_X = (ρ_X)_* ℒ_X = (ρ_X)_* Φ^*
ℒ_Y,
$$
there exists a natural morphism $α :φ^* ℱ_Y → ℱ_X$,
cf.~\cite[III~Rem.~9.3.1]{Ha77}.  The restriction of the morphism $α$ to the
open set $X°$ is clearly isomorphic.  We claim that it is isomorphic everywhere.
To this end, consider its determinant $\det α : φ^* \sM_Y → \sM_X$.  Using that
$φ^* \sM_Y ≅ \sM_X$, the determinant can be seen as a section in
$\sHom(\sM_X, \sM_X) ≅ 𝒪_X$, and hence as a function on $X$ that does not
vanish on $X°$.  But since the $φ$-exceptional set $E := X ∖ X°$ is
contracted to the small subvariety $Y∖ Y°$ of $Y$, the function in fact
cannot vanish anywhere.  It follows that the morphism $\det α$ is isomorphic,
and hence so is $α$.  This ends the proof of Theorem~\ref{thm:desc1}.  \qed

%
% Do not edit the following line.  The text is automatically updated by
% subversion.
%
\svnid{$Id: 05.tex 1145 2018-04-25 11:46:56Z kebekus $}

\section{Descent of vector bundles to klt spaces}\label{sec:4}
\approvals{Behrouz & yes \\ Daniel & yes \\ Stefan & yes \\ Thomas & yes}
\subversionInfo

In the present section, we will prove the following theorem, a simplified form
of which appeared as Theorem~\ref{thm:desc2simplified} in the introduction.

\begin{thm}[Descent of vector bundles to klt spaces]\label{thm:desc2}
  Let $f : X → Y$ be a birational, projective morphism of normal,
  quasi-projective varieties.  Assume $Y$ to be a klt space.  Then, the
  following holds.
  \begin{enumerate}
  \item\label{il:desc1} If $ℱ_X$ is any locally free, $f$-numerically flat sheaf
    on $X$, then there exists a locally free sheaf $ℱ_Y$ on $Y$ such that
    $ℱ_X ≅ f^* ℱ_Y$.

  \item\label{il:desc2} If $\bigl(ℱ_X,\, θ_X\bigr)$ is any locally free Higgs
    sheaf on $X$, where $ℱ_X$ is $f$-numerically flat, then there exists a
    locally free Higgs sheaf $\bigl(ℱ_Y,\, θ_Y\bigr)$ on $Y$ such that
    $\bigl(ℱ_X,\, θ_X\bigr) ≅ f^* \bigl(ℱ_Y,\, θ_Y\bigr)$.
  \end{enumerate}
\end{thm}

\subsection{Proof of Theorem~\ref{thm:desc2}: Setup and Notation}
\approvals{Behrouz & yes \\ Daniel & yes \\ Stefan & yes \\ Thomas & yes}

We maintain notation and assumptions of Theorem~\ref{thm:desc2} throughout the
present Section~\ref{sec:4}.  Choose an effective Weil $ℚ$-divisor $Δ_Y$ such
that $(Y, Δ_Y)$ is klt.  We denote by $Δ_X := f_*^{-1} Δ_Y$ the strict transform
of $Δ_Y$, and by $E_X := \Exc(f)$ the divisorial part of the $f$-exceptional
locus $\Exc(f)$.

\subsection{Proof of Statement~\ref*{il:desc1}}
\approvals{Behrouz & yes \\ Daniel & yes \\ Stefan & yes \\ Thomas & yes}

Consider a resolution of singularities, $π : \wtilde{X} → X$.  If we can show
that $π^* ℱ_X$ is of the form $(f ◦ π)^* ℱ_Y$ for a suitable sheaf $ℱ_Y$ on $Y$,
then $ℱ_X$ will be isomorphic to $f^* ℱ_Y$ by the projection formula.  We are
therefore free to replace $X$ by $\wtilde X$ and assume without loss of
generality that the following holds.

\begin{asswlog}
  The variety $X$ is smooth, the $f$-exceptional set equals $E_X$, and
  $\supp(Δ_X + E_X)$ is an snc divisor in $X$.
\end{asswlog}

\subsubsection*{Step 1: Factorisation via an $f$-relative MMP}
\approvals{Behrouz & yes \\ Daniel & yes \\ Stefan & yes \\ Thomas & yes}

We will factor the resolution $f: X → Y$ via a relative minimal model program of
$X$ over $Y$, cf.~the discussion in \cite[Sect.~23]{GKKP11}, whose organisation
we will follow closely.  By the definition of ``klt pair'' there exist effective
$f$-exceptional divisors $F$ and $G$ without common components such that
$\left⌊ F \right⌋ = 0$ and such that the following $ℚ$-linear
equivalence holds:
\begin{equation*}
  K_X + Δ_X + F \,\,\sim_ℚ\,\, f^*(K_Y+Δ_Y) + G.
\end{equation*}
For $ε ∈ (0,1) ∩ ℚ$ we let $Δ_ε := Δ_X+F+ε·E$.  For $0 < ε ≪ 1$ small enough,
the pair $(X,Δ)$ is klt.  Fix one such $ε$ and let $H ∈ ℚ\Div(X)$ be an
$f$-ample divisor such that $(X,Δ_{ε}+H)$ is still klt and $K_{X}+Δ_{ε}+H$ is
$f$-nef.  We may then run the $f$-relative $(X,Δ_\epsilon)$ minimal model
program with scaling of $H$, cf.\ \cite[Cor.~1.4.2]{BCHM10} to obtain a diagram
$$
\xymatrix{ %
  X \ar[d]_{f} \ar@{=}[r] & X_n \ar@{-->}[r]^{φ_n} \ar[d]_{f_n} & X_{n-1} \ar@{-->}[r]^{φ_{n-1}} \ar[d]_{f_{n-1}} & ⋯ \ar@{-->}[r]^{φ_1} & X_0 \ar[d]_{f_0} \\
  Y \ar@{=}[r] & Y \ar@{=}[r] & Y \ar@{=}[r] & ⋯ \ar@{=}[r] & Y %
}
$$
with the following properties.
\begin{enumerate}
\item\label{il:alpha} The spaces $X_i$ are $ℚ$-factorial.  Writing $Δ_{X_i}$ for
  the cycle-theoretic image of $Δ_ε$, the pairs $\bigl( X_i,\, Δ_{X_i}\bigr)$
  are klt.
\item\label{il:beta} The maps $φ_i$ are either divisorial contractions of
  $(K_{X_i}+Δ_{X_i})$-negative extremal rays in $R_i ⊆ \overline{NE}(X_i/Y)$ or
  flips associated to small contractions of such rays.
\item\label{il:gamma} The log-canonical divisor $K_{X_0} + Δ_{X_0}$ is
  $f_0$-nef, and it hence follows from the negativity lemma,
  \cite[Lem.~2.16.2]{GKKP11}, that the morphism $f_0$ is small and crepant,
  cf.~\cite[Claim 23.4]{GKKP11}.  In other words, $Δ_{X_0} = (f_0)^{-1}_* (Δ_Y)$
  and $K_{X_0}+Δ_{X_0} \; \sim_{ℚ} \; (f_0)^* \bigl( K_Y+Δ_Y \bigr)$.  As a
  consequence, also $-(K_{X_0}+Δ_{X_0})$ is $f_0$-nef.
\end{enumerate}

\subsubsection*{Step 2: Construction of bundles}
\approvals{Behrouz & yes \\ Daniel & yes \\ Stefan & yes \\ Thomas & yes}

Next, we construct vector bundles on the $X_i$.

\begin{claim}\label{claim:x1}
  There exist locally free sheaves $ℱ_{X_i}$ on the varieties $X_i$ such that
  the following holds.
  \begin{enumerate}
  \item\label{il:1} The sheaf $ℱ_{X_n}$ equals $ℱ_X$.
  \item\label{il:2} Given any index $i$, the sheaf $ℱ_{X_i}$ is
    $f_i$-numerically flat.
  \item\label{il:3} If $i>0$ is any index such that $ℱ_{X_{i-1}}$ is isomorphic
    to $(f_{i-1})^* 𝒢$ for a locally free sheaf $𝒢$ on $Y$, then
    $ℱ_{X_i} ≅ (f_i)^* 𝒢$.
  \end{enumerate}
\end{claim}
\begin{proof}[Proof of Claim~\ref{claim:x1}]
  We construct the vector bundles inductively.  Start by setting
  $ℱ_{X_n} := ℱ_X$.  Next, assume that we are given an index $i > 0$ for which
  vector bundles $ℱ_{X_n}$, …, $ℱ_{X_i}$ have already been constructed.  We
  consider the cases where $φ_i$ is a divisorial contraction and where it is a
  flip separately.

  \subsubsection*{Divisorial contraction}

  If $φ_i: X_i → X_{i-1}$ is a divisorial contraction, then
  $-(K_{X_i}+ Δ_{X_i})$ is $φ_i$-nef.  Theorem~\ref{thm:desc1} (``Descent of
  vector bundles from klt spaces'') hence proves the existence of a locally free
  sheaf $ℱ_{X_{i-1}}$ on $X_{i-1}$ such that $ℱ_{X_i} ≅ φ_i^* ℱ_{X_{i-1}}$.
  This isomorphism guarantees that Properties~\ref{il:2} and \ref{il:3} both
  hold.

  \subsubsection*{Flip}

  If $φ_i: X_i \dasharrow X_{i-1}$ is a flip, consider the associated ``flipping
  diagram'',
  $$
  \xymatrix{ %
    X_i \ar@{-->}@/^4mm/[rrrr]^{φ_i} \ar[rr]_{α} \ar@/_2mm/[rrd]_{f_i} && Z \ar[d]_{γ} && X_{i-1} \ar[ll]^{β} \ar@/^2mm/[lld]^{f_{i-1}} \\
    && Y,
  }
  $$
  where $α$ is obtained by contracting a $(K_{X_i} + Δ_{X_i})$-extremal ray,
  which implies as above that $-(K_{X_i}+ Δ_{X_i})$ is $α$-nef.  In this
  setting, Theorem~\ref{thm:desc1} again proves the existence of a locally free
  sheaf $ℱ_Z$ on $Z$ such that $ℱ_{X_i} ≅ α^* ℱ_Z$.  Properties~\ref{il:2}
  and \ref{il:3} will hold once we set $ℱ_{X_{i-1}} := β^* ℱ_Z$.  This finishes
  the proof of Claim~\ref{claim:x1}.
\end{proof}

\subsubsection*{Step 3: End of proof}
\approvals{Behrouz & yes \\ Daniel & yes \\ Stefan & yes \\ Thomas & yes}

To end the proof of Statement~\ref{il:desc1}, recall from \ref{il:gamma} that
$(X_0,\, Δ_{X_0})$ is klt and that $-(K_{X_0} + Δ_{X_0})$ is $f_0$-nef.  Since
$ℱ_{X_0}$ is $f_0$-numerically flat by \ref{il:2}, we may therefore apply
Theorem~\ref{thm:desc1} to obtain a locally free sheaf $ℱ_Y$ on $Y$ such that
$ℱ_{X_0} ≅ (f_0)^* ℱ_Y$.  A repeated application of Property~\ref{il:3} then
shows that the sheaves $ℱ_{X_1}$, $ℱ_{X_2}$, …, $ℱ_{X_n} = ℱ$ are all pull-backs
of $ℱ_Y$.  Statement~\ref{il:desc1} follows.

\subsection{Proof of Statement~\ref*{il:desc2}}
\approvals{Behrouz & yes \\ Daniel & yes \\ Stefan & yes \\ Thomas & yes}

For $• ∈ \{ X, Y \}$, set
$$
𝒜ⁱ_• := \sHom \bigl( ℱ_•,\, ℱ_• ⊗ Ω^{[i]}_• \bigr).
$$
A Higgs fields on $ℱ_Y$ is, by definition, a section of the sheaf $𝒜¹_Y$ such
that the induced section in $𝒜²_Y$ vanishes.  In our case, the sheaf $ℱ_Y$ is
locally free, which implies that both $𝒜¹_Y$ and $𝒜²_Y$ are reflexive.  To give
a Higgs field on $ℱ_Y$ it is therefore equivalent to give a Higgs field on the
restriction of $ℱ_Y$ to any big open subset of $Y$.  The Higgs field $θ_X$,
however clearly induces a Higgs field on the restriction of $ℱ_Y$ to the big
open set where $f^{-1}$ is well-defined and isomorphic.  The existence of $θ_Y$
follows.

It remains to show that
$f^* \bigl( ℱ_Y,\, θ_Y \bigr) ≅ \bigl( ℱ_X,\, θ_X \bigr)$.  In other words,
we need to show that the two sections $θ_X$ and
$f^* θ_Y ∈ H⁰ \bigl( X,\, 𝒜¹_X \bigr)$ agree.  They will clearly agree over the
open set $f^{-1}(Y°)$.  Since $ℱ_X$ is locally free and $𝒜¹_X$ therefore
reflexive, this suffices to show that they are the same.  This finishes the
proof of Theorem~\ref{thm:desc2}.  \qed

%
% Do not edit the following line.  The text is automatically updated by
% subversion.
%
\svnid{$Id: 06.tex 1153 2018-07-25 14:29:29Z kebekus $}

\section{The restriction theorem for semistable Higgs sheaves}

\approvals{Behrouz & yes \\ Daniel & yes \\ Stefan & yes \\ Thomas & yes}

\label{sect:app-oper}
\subversionInfo

The proof of the nonabelian Hodge correspondence uses the following restriction
theorem for (semi)stable Higgs sheaves, which generalises a number of earlier
results including \cite[Thm.~5.22]{GKPT15}.

\begin{thm}[Restriction of (semi)stable Higgs sheaves]\label{thm:restrictionA}
  Let $X$ be a normal, projective variety, $\dim X ≥ 2$, and let $H ∈ \Div(X)$
  be big and semiample.  Given any torsion free Higgs sheaf $(ℰ°, θ°)$ on
  $X_{\reg}$ that is semistable (resp.\ stable) with respect to $H$ in the sense
  of Definition~\ref{def:sHsl2}, there exists an integer $M ∈ ℕ^+$ satisfying
  the following conditions: If $\aB ⊆ |m·H|$ is any basepoint free linear system
  with $m > M$, then there exists a dense, open subset $\aB° ⊆ \aB$ such that
  the following properties hold for all $D ∈ \aB°$.
  \begin{enumerate}
  \item\label{il:r1a} The hypersurface $D$ is irreducible and normal, and
    $D_{\reg} = D ∩ X_{\reg}$.

  \item\label{il:r3a} The Higgs sheaf $(ℰ°, θ°)|_{D_{\reg}}$ is torsion free and
    semistable (resp.\ stable) with respect to $H|_D$.
  \end{enumerate}
\end{thm}

\begin{rem}[Algebraicity assumption]
  We stress that the Higgs sheaf $(ℰ°, θ°)$ of Theorem~\ref{thm:restrictionA} is
  assumed to be algebraic.
\end{rem}

\begin{rem}[Restriction theorem for sheaves on $X$]\label{rem:restaX}
  Recalling from Lemma~\ref{lem:owe1} that a Higgs sheaf on $X$ is semistable
  (resp.\ stable) if and only if its restriction to $X_{\reg}$ is semistable
  (resp.\ stable), Theorem~\ref{thm:restrictionA} immediately implies a
  restriction theorem for torsion free Higgs sheaves on $X$ that is more general
  than the results found in the literature.  In practical applications, the
  variety $X$ might admit a finite group action, and the linear system
  $\aB ⊆ |m·H|$ might be chosen to contain invariant divisors only.
\end{rem}

Theorem~\ref{thm:restrictionA} is shown in Section~\ref{ssec:potrA} below.  The
following corollary discusses the behaviour of semistability under pull-back.
It complements \cite[Sect.~5.6]{GKPT15}, where the ($G$-)stable case was
discussed.  \Publication{Its proof, spelled out in the arXiv version of this
  paper, \href{http://arxiv.org/abs/1711.08159}{arXiv:1711.08159}, applies
  Theorem~\ref{thm:restrictionA} repeatedly to cut down to a curve, where the
  result is classically known.}

\begin{cor}[Semistability under generically finite morphisms]\label{cor:suqema}
  Let $X$ and $Y$ be two projective, klt spaces.  Let $H ∈ \Div(X)$ be big and
  semiample, and let $f : Y → X$ be a surjective and generically finite
  morphism.  Let $(ℰ,θ)$ be a reflexive Higgs sheaf on $X$.
  \begin{itemize}
  \item If $ℰ$ is locally free, then the following are equivalent.
    \begin{enumerate}
    \item\label{il:Ap1} The Higgs bundle $(ℰ,θ)$ is semistable with respect to
      $H$.
    \item\label{il:Ap2} The Higgs bundle $f^*(ℰ,θ)$ is semistable with respect
      to $f^*H$.
    \end{enumerate}
    
  \item If $Y$ is smooth, then the following are equivalent.
    \begin{enumerate}
    \item\label{il:Ap3} The Higgs sheaf $(ℰ,θ)$ is semistable with respect to
      $H$.
    \item\label{il:Ap4} The Higgs sheaf $f^{[*]}(ℰ,θ)$ is semistable with
      respect to $f^*H$.  \Publication{\qed}
    \end{enumerate}
  \end{itemize}
\end{cor}

\begin{rem}
  One might wonder why the assumptions in Corollary~\ref{cor:suqema} are so much
  more restrictive compared to Theorem~\ref{thm:restrictionA}.  If $X$ is not
  klt, then pull-back of Higgs sheaves does not exist in general.  If the sheaf
  $ℰ$ is not locally free and $Y$ is not smooth, its pull-back is generally
  neither reflexive nor torsion free, and no good notion of ``stability'' is
  defined in this case.  Also, we do not know whether the reflexive pull-back
  $f^{[*]}ℰ$ carries a natural Higgs field in this case.
\end{rem}

%of Preprint

\subsection{Restriction theorem for sheaves with operators}
\approvals{Behrouz & yes \\ Daniel & yes \\ Stefan & yes \\ Thomas & yes}
\label{ssec:a1aA}

The following restriction theorem for sheaves with operators is a generalisation
of \cite[Thm.~A.3 in the arXiv version,
\href{http://arxiv.org/abs/1511.08822}{arXiv:1511.08822}]{GKPT15}.  It serves as
the main technical tool used in the proof of the restriction theorem for Higgs
sheaves, Theorem~\ref{thm:restrictionA}.

\begin{thm}[Restriction theorem for sheaves with operators]\label{thm:nrestrSWOA}
  Let $X$ be a normal, projective variety, $\dim X ≥ 2$, let $H$ be a big and
  semiample divisor on $X$, and $𝒲°$ be a reflexive sheaf on $X_{\reg}$.  Let
  $(ℰ°, θ°)$ be a torsion free sheaf on $X_{\reg}$ with a $𝒲°$-valued operator,
  and assume that $(ℰ°, θ°)$ is semistable (resp.\ stable) with respect to $H$.
  Then, there exists $M ∈ ℕ^+$ satisfying the following conditions: If
  $\aB ⊆ |m·H|$ is any basepoint free linear system with $m > M$, then there
  exists a dense, open subset $\aB° ⊆ \aB$ such that the following properties
  hold for all $D ∈ \aB°$.
  \begin{enumerate}
  \item\label{il:r1xb} The hypersurface $D$ is irreducible and normal, and
    $D_{\reg} = X_{\reg} ∩ D$.

  \item\label{il:r3xb} The sheaf $ℰ°|_{D_{\reg}}$ is torsion free, and
    $(ℰ°|_{D_{\reg}}, θ°|_{D_{\reg}})$ is semistable (resp.\ stable) with
    respect to $H|_D$, as a sheaf with a $𝒲|_{D_{\reg}}$-valued operator.
  \end{enumerate}
\end{thm}

\begin{rem}[Algebraicity and restriction theorem for sheaves on $X$]\label{rem:HN}
  As before, we underline that $(ℰ°, θ°)$ is assumed to be algebraic.  Also as
  before, Theorem~\ref{thm:nrestrSWOA} implies a restriction theorem for sheaves
  $ℰ$ with operators that are defined on all of $X$.  In particular, if $ℰ$ is a
  torsion free sheaf on $X$ (=sheaf equipped with the zero operator) that is not
  necessarily semistable, then $μ^{\max}_H(ℰ) = μ^{\max}_{H|_D}(ℰ|_D)$ for all
  $D ∈ \aB°$, and the Harder-Narasimhan filtration of $ℰ|_D$ equals the
  restriction of the Harder-Narasimhan filtration of $ℰ$, in symbols:
  $\HN^{•}_H(ℰ)|_D = \HN^{•}_H(ℰ|_D)$.
\end{rem}

\begin{proof}[Proof of Theorem~\ref{thm:nrestrSWOA}]
  Let $ι: X° → X$ be the inclusion map.  Set $ℰ = ι_*(ℰ°) $ and $𝒲 = ι_*(𝒲°)$.
  The sheaf $ℰ $ is torsion free; furthermore, $𝒲$ is reflexive and therefore
  embeds into a locally free sheaf $\sV$.  Composing morphisms as in
  Lemma~\ref{lem:owe2}, the operator $θ°$ induces an operator
  $τ° : ℰ° → ℰ° ⊗ \sV|_{X_{\reg}}$, and Lemma~\ref{lem:owe2} shows that
  $(ℰ°, τ°)$, considered as a sheaf with a $\sV|_{X_{\reg}}$-valued operator, is
  again semistable (resp.\ stable) with respect to $H$.  Pushing forward, we can
  extend the operator $τ°$, which a priori is only defined on $X_{\reg}$, to an
  operator $τ : ℰ → ℰ ⊗ \sV$.  Lemma~\ref{lem:owe1} then shows that $(ℰ, τ)$ is
  again semistable (resp.\ stable) with respect to $H$, when considered as a
  sheaf with a $\sV$-valued operator.  We claim the following.\CounterStep
  \begin{enumerate}
  \item\label{il:r4xb} There exists a number $M ∈ ℕ^+$ and for every $m > M$ and
    for every basepoint free linear system $\aB ⊆ |m·B|$ an open, dense
    subsystem $\aB°$ such that \ref{il:r1xb} holds, such that $ℰ|_D$ is torsion
    free, and such that $(ℰ|_D, θ|_D)$ is semistable (resp.\ stable) with
    respect to $H|_D$, as a sheaf with a $\sV|_D$-valued operator.
  \end{enumerate}
  As soon as \ref{il:r4xb} is true, Lemma~\ref{lem:owe1} asserts that
  $(ℰ|_{D_{\reg}}, τ|_{D_{\reg}}) = (ℰ°|_{D_{\reg}}, τ°|_{D_{\reg}})$ is
  semistable (resp.\ stable) with respect to $H$ as a sheaf with an
  $\sV|_{D_{\reg}}$-valued operator.  A final application of
  Lemma~\ref{lem:owe2} then shows that $(ℰ°|_{D_{\reg}}, θ°|_{D_{\reg}})$ is
  semistable (resp.\ stable) with respect to $H$, which would end the proof.

  It remains to show \ref{il:r4xb}.  To this end, observe that for any
  resolution of singularities, $π : \wtilde{X} → X$, the operator $τ$ induces an
  operator on the torsion free pullback of $ℰ$ to $\wtilde{X}$,
  $$
  \wtilde{τ} : \factor{π^*ℰ}{\tor} → \factor{π^*(ℰ⊗\sV)}{\tor} = \factor{π^*
    ℰ}{\tor} ⊗ π^*\sV.
  $$
  Replacing $X$ by a suitable resolution and replacing $(ℰ, τ)$ by
  $(π^* ℰ/\tor, \wtilde{τ})$, we may assume without loss of generality that $X$
  is smooth and that $ℰ$ is locally free.  Now, if $(ℰ, τ)$ is stable with
  respect to $H$, the stability claim of Theorem~\ref{thm:nrestrSWOA} has been
  shown already for every $D ∈ \aB$ fulfilling \ref{il:r1xb} and having the
  property that $ℰ|_D$ is torsion free, see \cite[Thm.~9]{MR3314517} and
  \cite[Thm.~A.3 in the arXiv version]{GKPT15}.  These two assumptions can be
  guaranteed for $D$ belonging to an open subsystem $\aB°$ by the classical
  Bertini theorem and by \cite[Thm.~12.2.1]{EGA4-3}, respectively.  In case
  where $(ℰ, τ)$ only semistable, the assumption that $\sV$ is a locally free
  guarantees the existence of a Jordan-Hölder filtration, which presents
  $(ℰ, τ)$ as a repeated extension of stable sheaves with $\sV$-valued operator
  of equal slope.  The claim then follows by induction on the length of the
  filtration.
\end{proof}

\subsection{Proof of Theorem~\ref{thm:restrictionA}}
\approvals{Behrouz & yes \\ Daniel & yes \\ Stefan & yes \\ Thomas & yes}
\label{ssec:potrA}

For convenience of notation, write
$n := \dim X$, let $ι : X_{\reg} → X$ be the inclusion map, and let
$ℰ := ι_* ℰ°$ be the torsion free extension of $ℰ°$ to $X$.  Twisting the Higgs
sheaf $(ℰ°, θ°)$ with a sufficiently positive line bundle and noticing that
semistability considerations are unaffected by this operation, we may also
assume that the following holds.

\begin{asswlog}\label{ass:0}
  The slope of $ℰ$ is positive, $μ_H (ℰ) > 0$.
\end{asswlog}

\begin{choice}
  Choose $M ≫ 0$ large enough so that the following holds.
  \begin{enumerate}
  \item\label{ilx:0aa} If $m > M$ and $\aB ⊂ |m H|$ is base point free, then
    Theorem~\ref{thm:nrestrSWOA} and Remark~\ref{rem:HN} apply to yield a dense,
    open set $\aB' ⊂ \aB$ such that for every $D ∈ \aB'$, we have
    $μ^{\max}_H(ℰ) = μ^{\max}_{H|_D}(ℰ|_D)$.
   
  \item\label{ilx:0b} The number $M$ is larger than the number given by
    Theorem~\ref{thm:nrestrSWOA} for $(ℰ°, θ°)$ as a sheaf with an
    $Ω¹_{X_{\reg}}$-valued operator.
    
  \item\label{ilx:1c} We have $[M·H]^n > nr·μ^{\max}_H(ℰ)$.
  \end{enumerate}
\end{choice}

\subsubsection*{Step 1: Argument by contradiction}
\approvals{Behrouz & yes \\ Daniel & yes \\ Stefan & yes \\ Thomas & yes}

Assume for the remainder of the proof that we are given a number $m > M$ and a
basepoint free linear system $\aB ⊆ |m·H|$.  Let $\aB°$ be the intersection of
the open subsets given by the two applications of \ref{ilx:0aa} and \ref{ilx:0b}
above, and let $D ∈ \aB°$ be any element.  We aim to show that
$(ℰ°, θ°)|_{D_{\reg}}$ is semistable (resp.\ stable) with respect to $H$.  We
argue by contradiction and assume that this is not the case.

\begin{assumption}\label{ass:1}
  There exists a generically Higgs-invariant, saturated subsheaf
  $0 ⊊ ℱ_{D_{\reg}} ⊊ ℰ°|_{D_{\reg}}$ with $μ_{H|_D}(ℱ_{D_{\reg}}) > μ_H(ℰ)$
  (resp.~$≥$ instead of $>$).
\end{assumption}

\subsubsection*{Step 2: Cutting down}
\approvals{Behrouz & yes \\ Daniel & yes \\ Stefan & yes \\ Thomas & yes}

Repeated applications of Theorem~\ref{thm:nrestrSWOA} (and Remark~\ref{rem:HN})
allow to find an increasing sequence of numbers $M < m ≤ m_2 ≤ ⋯ ≤ m_{n-1}$ and
hyperplanes $D_i ∈ |m_i·H|$ such that the associated intersection
$C := D ∩ D_2 ∩ ⋯ ∩ D_{n-1}$ has the following properties.
\begin{enumerate}\CounterStep
\item The scheme $C$ is a smooth curve, and entirely contained in $X_{\reg}$.
  The sheaves $ℰ|_C$ and $ℱ_{D_{\reg}}|_C$ are torsion free and hence locally
  free.  The natural morphism $ℱ_{D_{\reg}}|_C → ℰ|_C$ is an injection.
  
\item\label{ilx:0ab} We have $μ^{\max}_H(ℰ) = μ^{\max}_{H|_C}(ℰ|_C)$.
  
\item\label{ilx:1b} Because of \ref{ilx:0b}, the restriction $(ℰ|_C, θ°|_C)$ is
  semistable (resp.\ stable) as a sheaf with an $Ω¹_X|_C$-valued operator.
\end{enumerate}

\subsubsection*{Step 3: Computation}
\approvals{Behrouz & yes \\ Daniel & yes \\ Stefan & yes \\ Thomas & yes}

In the following, write $ℱ_C := ℱ_{D_{\reg}}|_C$ and consider the associated
sequence\CounterStep
\begin{equation}\label{eq:qwen}
  \xymatrix{
    0 \ar[r] & ℱ_C \ar[r] & ℰ|_C \ar[r]^(.4){q} & 𝒬 \ar[r] & 0.
  }
\end{equation}
Let $N^*_{C/X}$ denote the conormal bundle of $C$ in $X$.  If
$ℬ ⊆ 𝒬 ⊗ N^*_{C/X}$ is any coherent subsheaf of positive rank, then we will show
in this step that
\begin{equation}\label{eq:hklhi}
  \deg_C(ℬ) ≤ (n-1)r·μ^{\max}_H(ℰ) - [m·H]^n.
\end{equation}
To prove \eqref{eq:hklhi}, consider first any coherent subsheaf $𝒜 ⊆ 𝒬$ of
positive rank.  We obtain from \eqref{eq:qwen} an exact sequence
$0 → ℱ_C → q^{-1}𝒜 \overset{q}{→} 𝒜 → 0$, which allows to estimate the degree of
$𝒜$ as follows,
\begin{align*}
  \deg_C (𝒜) & = \deg_C (q^{-1} 𝒜) - \deg_C (ℱ_C) ≤ \deg_C (q^{-1}𝒜) && \text{since } μ_{H|_C}(ℱ_C) ≥ μ_H(ℰ) > 0 \\
  & ≤ \rank (q^{-1} 𝒜) · μ_{H|_C} (q^{-1} 𝒜) && \text{definition of slope}\\
  & ≤ \rank (ℰ) · μ^{\max}_H (ℰ) && \text{Item~\ref{ilx:0ab}}
\end{align*}
In order to apply this inequality to the problem at hand, recall that $C$ is
constructed as a complete intersection.  The normal bundle of $C$ in $X$ is
hence described as
$$
N_{C/X} = ⊕_i N_{D_i/X}|_C, \qquad \text{where } N_{D_i/X}|_C ≅ 𝒪_C(m·H|_C) \text{ for all } i.
$$
We can therefore view $𝒜 := ℬ(m·H|_C)$ as a subsheaf of $𝒬^{⊕ n-1}$.  An
induction using the inequality obtained above then shows the following, which
immediately implies~\eqref{eq:hklhi},
$$
\deg_C(ℬ) + \rank (ℬ)·m·\deg_C (H|_C) = \deg_C ℬ(m·H|_C)
\overset{\text{Induction}}{≤} (n-1)r·μ^{\max}_H (ℰ).
$$

\subsubsection*{Step 4: End of proof}
\approvals{Behrouz & yes \\ Daniel & yes \\ Stefan & yes \\ Thomas & yes}

Consider the operator $θ°|_C$ and its restriction
$θ_{ℱ} : ℱ_C → ℰ|_C ⊗ Ω¹_X|_C$.  The target of $θ_{ℱ}$ appears in the following
commutative square,
$$
\xymatrix{ %
           & & ℰ|_C ⊗ Ω¹_X|_C \ar@{->>}[r] \ar@{->>}[d]_{q ⊗ \Id} & ℰ|_C ⊗ Ω¹_C \ar@{->>}[d] \\
  0 \ar[r] & 𝒬 ⊗ N^*_{C/X} \ar[r]_{α} & 𝒬 ⊗ Ω¹_X|_C \ar[r]_{β} & 𝒬 ⊗ Ω¹_C \ar[r] & 0.
}
$$
Assumption~\ref{ass:1} and Item~\ref{ilx:1b} guarantees that $ℱ_C$ is not
invariant with respect to $θ°|_C$.  The composed map $(q⊗ \Id) ◦ θ_{ℱ}$ is
therefore \emph{not} the zero morphism.  But since $ℱ_C$ is a Higgs-invariant
subsheaf of $(ℰ°, θ°)|_C$ by assumption, the morphism $β ◦ (q⊗\Id) ◦ θ_{ℱ}$
\emph{is} zero.  In summary, we obtain a non-trivial morphism
$τ : ℱ_C → 𝒬 ⊗ N^*_{C/X}$.  But then, we compute degrees as follows,
\begin{align*}
  (n-1)r·μ^{\max}_H(ℰ) - [m·H]^n & ≥ \deg (\img τ) && \text{Inequality~\eqref{eq:hklhi}}\\
                                  & = \deg_C (ℱ_C) - \deg_C (\ker τ) \\
                                  & ≥ - \deg_C(\ker τ) && \text{Assumptions~\ref{ass:0}, \ref{ass:1}} \\
                                  & ≥ - \rank (\ker τ) · μ^{\max}_H (ℰ) && \text{Item~\ref{ilx:0ab}} \\
                                  & ≥ -r·μ^{\max}_H (ℰ) && \text{Assumption~\ref{ass:0}.}
\end{align*}
We obtain a contradiction to \ref{ilx:1c}, which finishes the proof of
Theorem~\ref{thm:restrictionA}.  \qed

%
% Do not edit the following line.  The text is automatically updated by
% subversion.
%
\svnid{$Id: 07.tex 1151 2018-07-24 11:30:18Z kebekus $}

\section{Ascent of semistable Higgs bundles}
\label{ssec:ascend}
\subversionInfo
\approvals{Behrouz & yes \\ Daniel & yes \\ Stefan & yes \\ Thomas & yes}

The proof of the nonabelian Hodge correspondence on klt spaces relies on the
following result, which can be seen as an inverse to the descent results
obtained in the first part of this paper.  It asserts that the pull-back of a
Higgs sheaf in $\Higgs_X$ to any resolution of singularities is again in
$\Higgs_•$.
 
\begin{thm}[Ascent of semistable Higgs bundles]\label{thm:bphgs}
  Let $X$ be a projective, klt space and let $π : \wtilde{X} → X$ be a
  resolution of singularities.  If $(ℰ,θ) ∈ \Higgs_X$, then
  $π^*(ℰ,θ) ∈ \Higgs_{\wtilde{X}}$.
\end{thm}

The following is an almost immediate consequence of Theorem~\ref{thm:bphgs} and
Theorem~\ref{thm:desc2}, Item~\ref{il:desc2}; see also Fact~\ref{fact:dfsj}
below.

\begin{cor}[Boundedness of Higgs sheaves]\label{cor:dfsj1}
  Let $X$ be a projective, klt space and let $r ∈ ℕ^+$ be any number.  Let $\aF$
  be the family of locally free Higgs sheaves $(ℱ,Θ) ∈ \Higgs_X$ with
  $\rank ℱ = r$.  Then, the family $\aF$ is bounded.  \qed
\end{cor}

Theorem~\ref{thm:bphgs} will be shown in
Sections~\ref{subsect:higherx}--\ref{ssec:gydydg} below.  The main difficulty
here is that it is not clear from the outset that a pull-back of semistable
Higgs bundle via the resolution map is again semistable with respect to an
\emph{ample} divisor.  It turns out that the assumption of vanishing Chern
classes is sufficient to resolve this problem in dimension two.  The
higher-dimensional case will be reduced to the surface case by restriction
techniques.

\subsection{Preparation for the proof of Theorem~\ref{thm:bphgs}: Boundedness}
\label{subsect:higherx}
\approvals{Behrouz & yes \\ Daniel & yes \\ Stefan & yes \\ Thomas & yes}

We will use the following iterated Bertini-type theorem for bounded families of
Higgs sheaves, generalising \cite[Cor.~5.3]{GKP13}, where the same result was
shown for reflexive sheaves without Higgs field that are defined on a projective
variety, and not just on some big open subset.

\begin{prop}[Iterated Bertini-type theorem for bounded families on $X_{\reg}$]\label{prop:restr}
  Let $X$ be a normal, projective variety of dimension $n ≥ 2$, let
  $H ∈ \Div(X)$ be ample and let $X° ⊆ X_{\reg}$ be a big, open subset.  Let
  $(ℰ_{X°}, θ_{ℰ_{X°}})$ be a locally free Higgs sheaf on $X°$, and let $\aF°$
  be a bounded family of locally free Higgs sheaves on $X°$.  If $m ≫ 0$ is
  large enough and if $(D_1, D_2, …, D_{n-1}) ∈ |m·H|^{⨯ n-1}$ is a general
  tuple with associated complete intersection curve $C := D_1 ∩ ⋯ ∩ D_{n-1}$,
  then $C ⊂ X°$ is smooth and the following holds for all Higgs sheaves
  $(ℱ_{X°}, θ_{ℱ_{X°}})$ in $\aF°$:
  $$
  (ℱ_{X°}, θ_{ℱ_{X°}}) ≅ (ℰ_{X°}, θ_{ℰ_{X°}}) \quad ⇔ \quad
  (ℱ_{X°}, θ_{ℱ_{X°}})|_C ≅ (ℰ_{X°}, θ_{ℰ_{X°}})|_C.
  $$
\end{prop}
\begin{proof}
  As a first step in the proof, we extend all relevant sheaves from $X°$ to $X$.
  To this end, let $ℰ_X$ be the unique reflexive sheaf on $X$ whose restriction
  to $X°$ is $ℰ_{X°}$.  Doing the same with the sheaves that appear in $ℱ°$,
  observe that the family
  \begin{multline*}
    \aF := \bigl\{\text{isomorphism classes of reflexive $ℱ_X$ on $X$ such that there exists}\\
    \text{a member $(ℱ_{X°}, θ_{ℱ_{X°}})$ in $\aF°$ with $ℱ_X|_{X°} ≅ ℱ_{X°}$} \bigr\}
  \end{multline*}
  is likewise bounded.  This fact is crucial in the proof of the following two
  claims.

  \begin{claim}\label{claim:1}
    Setting $D°_1 := D_1 ∩ X°$, the restriction maps
    \begin{align}
      \label{il:cl1a} \Hom_{X°} \bigl(ℱ_{X°},\, ℰ_{X°}\bigr) & → \Hom_{D_1°} \bigl(ℱ_{X°}|_{D_1°},\, ℰ_{X°}|_{D_1°}\bigr) \\
      \label{il:cl1b} \Hom_{X°} \bigl(ℱ_{X°},\, ℰ_{X°}\bigr) & → \Hom_C \bigl(ℱ_{X°}|_C,\, ℰ_{X°}|_C \bigr) \\
      \label{il:cl1c} \Hom_{X°} \bigl(ℱ_{X°},\, ℱ_{X°} ⊗ Ω¹_{X°} \bigr) & → \Hom_C \bigl(ℱ_{X°}|_C,\, ℱ_{X°}|_C ⊗ Ω¹_{X°}|_C \bigr)
    \end{align}
    are isomorphic, for all $(ℱ_{X°}, θ_{ℱ_{X°}})$ in $\aF°$.
  \end{claim}
  \begin{proof}[Proof of Claim~\ref{claim:1}]
    For brevity, we consider \eqref{il:cl1a} only and leave the rest to the
    reader.  Since $m ≫ 0$ is assumed to be large, we have vanishing
    $$
    h⁰ \bigl( X,\, \sHom(ℱ_X, ℰ_X) ⊗ 𝒥_D \bigr) = h¹ \bigl( X,\, \sHom(ℱ_X, ℰ_X)
    ⊗ 𝒥_D \bigr) = 0
    $$
    for all members $ℱ_X$ of the bounded family $\aF$.  In fact, vanishing for
    $h¹$ follows from \cite[Exp.~XII, Prop.~1.5]{SGA2} since $X$ is normal,
    which implies by \cite[Prop.~1.3]{MR597077} that the sheaves
    $\sHom(ℱ_X, ℰ_X) ⊗ 𝒥_D$ have $\operatorname{depth} ≥ 2$ at every point of
    $X$.  As a consequence, we obtain that the natural restriction maps,
    \begin{equation}\label{eq:xx}
      \Hom_X \bigl(ℱ_X,\, ℰ_X \bigr) → H⁰ \bigl( D_1,\, \sHom(ℱ_X, ℰ_X)|_{D_1}
      \bigr)
    \end{equation}
    are isomorphic for all $ℱ_X$ in $\aF$.  In order to relate \eqref{eq:xx} to
    the problem at hand, use boundedness of $\aF$ again and recall from
    \cite[Cor.~1.1.14]{MR2665168} that the sheaves $ℱ_X|_{D_1}$, $ℰ_X|_{D_1}$
    and $\sHom(ℱ_X, ℰ_X)|_{D_1}$ are reflexive on the normal variety $D_1$, for
    all $ℱ_X$ in $\aF$.  As a consequence, we find that the natural restriction
    morphisms,
    \begin{align*}
      \Hom_X \bigl(ℱ_X,\, ℰ_X \bigr) & → \Hom_{X°} \bigl(ℱ_{X°},\, ℰ_{X°} \bigr) \\
      H⁰ \bigl( D_1,\, \sHom(ℱ_X, ℰ_X)|_{D_1} \bigr) & → \Hom_{D°_1} \bigl(ℱ_{X°}|_{D°_1},\, ℰ_{X°}|_{D°_1} \bigr)
    \end{align*}
    are all isomorphic.  \qedhere \quad (Claim~\ref{claim:1})
  \end{proof}

  \begin{claim}\label{claim:2}
    If $ι_{C} : C → X°$ denotes the inclusion, then the composed morphisms
    $δ_{ℱ_{X°}}$ of the following diagrams,
    \begin{equation}\label{eq:iop}
      \begin{aligned}
        \xymatrix{ %
          && \Hom_{C} \bigl(ℱ_{X°}|_C,\, ℱ_{X°}|_C ⊗ N^*_{C/X} \bigr) \ar@{^(->}[d] \\
          \Hom_{X°} \bigl(ℱ_{X°},\, ℱ_{X°} ⊗ Ω¹_{X°} \bigr) \ar[rr]^(.45){\text{restriction}}_(.45){\text{isom.~by \eqref{il:cl1c}}} \ar@/_3mm/[rrd]_(.4){δ_{ℱ_{X°}}} && \Hom_{C} \bigl(ℱ_{X°}|_C,\, ℱ_{X°}|_C ⊗ Ω¹_{X°}|_C \bigr) \ar[d]^{\Id ⊗ dι_C} \\
          && \Hom_C \bigl(ℱ_{X°}|_C,\, ℱ_{X°}|_C ⊗ Ω¹_C \bigr), %
        }
      \end{aligned}
\end{equation}
    are injective, for all $(ℱ_{X°}, θ_{ℱ_{X°}})$ in $\aF°$.
  \end{claim}
  \begin{proof}[Proof of Claim~\ref{claim:2}]
    By exactness of the vertical sequence, it suffices to show that
    \begin{equation}\label{eq:jhh}
      \Hom_{C} \bigl(ℱ_{X°}|_C,\, ℱ_{X°}|_C ⊗ N^*_{C/X} \bigr) = \{ 0 \} \text{, for all $ℱ_{X°}$ in $\aF°$}.
    \end{equation}
    To this end, observe that the normal bundle of $C$ in $X$ is very
    positive.  In fact, we have $N_{C/X} ≅ 𝒪_X(m·H)^{⊕ n-1}|_C$ by
    construction.  Since $m ≫ 0$ is assumed to be large, it follows from
    semicontinuity of the Harder-Narasimhan filtration,
    \cite[Thm.~2.3.2]{MR2665168}, that
    $$
    μ^{\max}_X\bigl( \bigl(ℱ_X^* ⊗ ℱ_X ⊗ 𝒪_X(-m·H)\bigr)^{**}
    \bigr) < 0 \text{, for all $ℱ_X$ in $\aF$}.
    $$
    In particular, Flenner's version of the Mehta-Ramanathan theorem,
    \cite[Thm.~7.1.1]{MR2665168}, implies that the sheaf
    $$
    \sHom_{C} \bigl(ℱ_{X°}|_C,\, ℱ_{X°}|_C ⊗ N^*_{C/X} \bigr)
    $$
    has no section.  \qedhere \quad (Claim~\ref{claim:2})
  \end{proof}
  \CounterStep

  Coming back to the proof of Proposition~\ref{prop:restr}, we need to show the
  implication ``$\Leftarrow$'' only.  Assume that an element
  $(ℱ_{X°}, θ_{ℱ_{X°}})$ in $\aF°$ and an isomorphism of Higgs sheaves,
  $$
  φ_C : (ℱ_{X°}, θ_{ℱ_{X°}})|_C → (ℰ_{X°}, θ_{ℰ_{X°}})|_C
  $$
  are given.  In other words, denoting the obvious inclusion map by
  $ι_C: C → X°$, we are given an isomorphism of sheaves,
  $φ_C : ℱ_{X°}|_C → ℰ_{X°}|_C$ and a commutative diagram
  \begin{equation}\label{eq:hmmn}
    \begin{aligned}
      \xymatrix{ %
        ℱ_C \ar[rrr]^(.4){(\Id_{ℱ_C} ⊗ dι_C) \,◦\, (θ_{ℱ_{X°}}|_C)} \ar[d]_{φ_C}^{≅} &&& ℱ_C ⊗ Ω¹_C \ar[d]^{φ_C ⊗ \Id_{Ω¹_C}} \\
        ℰ_C \ar[rrr]_(.4){(\Id_{ℰ_C} ⊗ dι_C) \,◦\, (θ_{ℰ_{X°}}|_C)} &&& ℰ_C ⊗ Ω¹_C \\
      }
    \end{aligned}
    \quad\text{where}\quad
    \begin{matrix}
      ℱ_C := ℱ_{X°}|_C \\
      ℰ_C := ℰ_{X°}|_C.
    \end{matrix}
  \end{equation}
  Item~\eqref{il:cl1b} guarantees that the sheaf morphism $φ_C$ extends in a
  unique manner to an isomorphism $φ_{X°} : ℱ_{X°} → ℰ_{X°}$.  This way, we
  obtain two Higgs fields on $ℱ_{X°}$, namely
  $$
  θ_{ℱ_{X°}} \qquad\text{and}\qquad θ'_{ℱ_{X°}} := (φ_{X°} ⊗ \Id_{Ω¹_{X°}})^{-1}
  ◦ θ_{ℰ_{X°}} ◦ φ_{X°}.
  $$
  Diagram~\eqref{eq:hmmn} implies that both Higgs fields $θ_{ℱ_{X°}}$ and
  $θ'_{ℱ_{X°}}$ restrict to the same Higgs field on $ℱ_{X°}|_C$.  But then
  Claim~\ref{claim:2} guarantees that the two Higgs fields $θ_{ℱ_{X°}}$ and
  $θ'_{ℱ_{X°}}$ are in fact equal.  In other words, the sheaf isomorphism
  $φ_{X°}$ is an isomorphism of Higgs sheaves.
\end{proof}

We use the following boundedness result for families of Higgs sheaves.  Its
proof is very similar to \cite[Claim~8.5 and proof]{GKPT15} and therefore
omitted.  We emphasise that unlike Corollary~\ref{cor:dfsj1}, the result
formulated here does not depend on Theorem~\ref{thm:bphgs} and can therefore be
used in the proof of that theorem.

\begin{fact}[Boundedness of Higgs sheaves]\label{fact:dfsj}
  Let $X$ be a klt space, let $r ∈ ℕ^+$ be any number, and let
  $π: \wtilde{X} → X$ be a resolution of singularities.  Let $\aF$ be the family
  of locally free Higgs sheaves $(ℱ,Θ)$ on $X$, that satisfy $\rank ℱ = r$ and
  have the additional property that $π^*(ℱ, Θ) ∈ \Higgs_{\wtilde{X}}$.  Then,
  the family $\aF$ is bounded.  \qed
\end{fact}

\subsection{Preparation for the proof of Theorem~\ref{thm:bphgs}: Vanishing of Chern classes}
\approvals{Behrouz & yes \\ Daniel & yes \\ Stefan & yes \\ Thomas & yes}

The following observation is rather standard, but used several times in this
section.

\begin{lem}[Description of bundles with pull-back in $\Higgs_{\wtilde{X}}$]\label{lem:xxs}
  Let $X$ be a projective, klt space and let $π : \wtilde{X} → X$ be a
  resolution of singularities.  If $(ℰ,θ)$ is any locally free Higgs sheaf on
  $X$ such that $π^*(ℰ,θ) ∈ \Higgs_{\wtilde{X}}$, then all Chern classes
  $c_i(ℰ) ∈ H^{2i}\bigl( X,\, ℚ \bigr)$ vanish and $(ℰ,θ)$ is semistable with
  respect to any nef divisor $N ∈ \Div(X)$.
\end{lem}
\begin{proof}
  To prove that $(ℰ,θ)$ is semistable with respect to any nef divisor, it is
  equivalent to show that $π^*(ℰ,θ)$ is semistable with respect to any nef
  divisor on $\wtilde{X}$, \cite[Prop.~5.20]{GKPT15}.  Argue by contradiction
  and assume that $π^*(ℰ,θ)$ is \emph{not} semistable with respect some nef
  $N ∈ \Div(\wtilde{X})$.  As non-semistability is an open condition, if
  $A ∈ \Div(\wtilde{X})$ is ample and $δ ∈ ℚ^+$ sufficiently small, then
  $π^*(ℰ,θ)$ fails to be semistable with respect to the ample divisor $N + δ·A$.
  This contradicts the fact that $π^*(ℰ,θ)$ is semistable with respect to
  \emph{any} ample divisor on $\wtilde{X}$, cf.~\cite[Thm.~1.3]{MR2231055}.
  Semistability of $(ℰ,θ)$ with respect to any nef divisor follows as desired.

  Concerning Chern classes, recall from \cite[Rem.~on p.~36]{MR1179076} that the
  $\cC^{∞}$-bundle $\wtilde{E}$ on $\wtilde{X}^{an}$ underlying $\wtilde{ℰ}$ is
  induced by a local system.  Takayama has shown in
  \cite[Prop.~2.1]{Takayama2003} that $X^{an}$ can be covered by contractible,
  open subsets $U_i$ with simply connected inverse images $(π^{an})^{-1}(U_i)$.
  We infer that the $\cC^{∞}$-bundle $E$ on $X^{an}$ underlying
  $ℰ ≅ π_*\, \wtilde{ℰ}$ is again induced by a local system.  But then, a
  classical result of Deligne-Sullivan, \cite{MR0397729}, implies that there
  exists a finite, étale cover $γ: \what{X} → X$ where $(γ^{an})^* \cE$ is
  trivial.  Vanishing of Chern classes downstairs then follows from the Leray
  spectral sequence and from the splitting of the natural map
  $ℚ_X → (γ^{an})_*\, ℚ_{\what{X}}$.
\end{proof}

\subsection{Preparation for the proof of Theorem~\ref{thm:bphgs}: Numerical flatness of Higgs sheaves}
\approvals{Behrouz & yes \\ Daniel & yes \\ Stefan & yes \\ Thomas & yes }

The following observation allows to apply the results obtained in
Section~\ref{sec:4} above to Higgs sheaves on resolutions of klt spaces.

\begin{prop}[Numerical flatness of Higgs sheaves]\label{prop:Higgs_are_num_flat}
  Let $π: \wtilde{X} → X$ be a resolution of a klt space, and let $\aE$ be any
  local system on $\wtilde X$.  Then, the corresponding locally free Higgs sheaf
  $(ℰ, θ)$ is $π$-numerically flat.
\end{prop}
\begin{proof}
  Given any smooth, projective curve $C$ and any morphism $γ : C → \wtilde{X}$
  such that $π◦γ$ is constant, we need to show that both $γ^* ℰ$ and its dual
  are nef.  To this end, recall from \cite[p.~827]{Takayama2003} that every
  fibre of $π$ admits a small, simply connected neighbourhood $U$, open in the
  Euclidean topology.  The local system $γ^* (π^* \aE)$ is therefore trivial.
  As Simpson's nonabelian Hodge correspondence is functorial in morphisms
  between manifolds, \cite[Rem.\ on p.~36]{MR1179076}, it follows that the
  pull-back $γ^* (ℰ, θ)$ corresponds to the trivial local system, and is
  therefore trivial itself; in particular, the pullback $γ^* \wtilde{ℰ}$ and its
  dual are both nef, as desired.
\end{proof}

\subsection{Proof of Theorem~\ref*{thm:bphgs} if $X$ is a surface}
\approvals{Behrouz & yes \\ Daniel & yes \\ Stefan & yes \\ Thomas & yes }
\label{ssec:ptfdfcv}

The following proposition immediately implies Theorem~\ref{thm:bphgs} (``Ascent
of semistable Higgs bundles'') in case where $X$ is a surface.

\begin{prop}[Independence of the polarisation for surfaces]\label{prop:sttc}
  Let $X$ be a smooth projective surface.  Given a locally free Higgs sheaf
  $(ℰ,θ)$ on $X$, the following statements are equivalent.
  \begin{enumerate}
  \item\label{il:y1} There exists a big and nef divisor $N$ such that
    $ch_1(ℰ)·N = ch_2(ℰ) = 0$ and such that $(ℰ,θ)$ is semistable with respect
    to $N$.
  \item\label{il:y2} All Chern classes of $c_i(ℰ) ∈ H^{2i}\bigl( X,\, ℚ \bigr)$
    vanish and $(ℰ,θ)$ is semistable with respect any ample divisor on $X$.
  \end{enumerate}
\end{prop}
\begin{proof}
  Using the fact that big and nef divisors are limits of ample divisors, the
  direction \ref{il:y2} $⇒$ \ref{il:y1} is immediate.  We will therefore
  consider the converse direction \ref{il:y1} $⇒$ \ref{il:y2} for the remainder
  of the proof.  Set $r := \rank ℰ$.

  \subsubsection*{Step 1.  Proof in case where $(ℰ,θ)$ is stable with respect to $N$}
  \approvals{Behrouz & yes \\ Daniel & yes \\ Stefan & yes \\ Thomas & yes}

  Using the assumption on stability with respect to a big and nef class, recall
  from \cite[Thm.~2.1]{MR1954067} or \cite[Thm.~6.1]{GKPT15} that $ℰ$ satisfies
  the Bogomolov-Gieseker inequality,
  \begin{equation}\label{eq:bginef}
    2r · c_2(ℰ) - (r-1)· c_1(ℰ)² ≥ 0.
  \end{equation}
  Substituting the assumption $ch_2(ℰ)=0$ into \eqref{eq:bginef}, we obtain that
  $c_1(ℰ)² ≥ 0$.  On the other hand, the assumption $ch_1(ℰ)·N = 0$, together
  with the Hodge index theorem implies that $c_1(ℰ)² ≤ 0$, with equality if and
  only if $c_1(ℰ)=0$.  We obtain that both $c_1(ℰ)$ and $c_2(ℰ)$ vanish, as
  desired.
  
  Next, let $H ⊂ S$ be any ample divisor.  By openness of stability,
  \cite[Prop.~4.17]{GKPT15}, there exists $ε ∈ ℚ^+$ such that $(ℰ,θ)$ is stable
  with respect to the ample class $(N+εH)$.  Lemma~\ref{lem:xxs} then asserts
  that $(ℰ,θ)$ is stable with respect to any ample divisor.  This finishes the
  proof in the stable case.

  \subsubsection*{Step 2.  Proof in general -- setup}
  \approvals{Behrouz & yes \\ Daniel & yes \\ Stefan & yes \\ Thomas & yes}

  Since $X$ is smooth, there exists a Jordan-Hölder filtration for the Higgs
  bundle $(ℰ,θ)$ and the nef polarisation $N$.  More precisely, we obtain a
  filtration
  $$
  0 = ℰ_0 ⊊ ℰ_1 ⊊ ⋯ ⊊ ℰ_{t-1} ⊊ ℰ_t = ℰ,
  $$
  with the following properties.
  \begin{enumerate}
  \item Each of the sheaves $ℰ_i$ is saturated in $ℰ_{i+1}$, hence reflexive,
    hence locally free since $X$ is a surface.
  \item The torsion free quotients $𝒬_i := ℰ_{i}/ℰ_{i-1}$ satisfy
    $c_1(𝒬_i)·N = 0$, and inherit Higgs fields $τ_i$ making $(𝒬_i, τ_i)$ stable
    with respect to $N$.
  \end{enumerate}
  Moreover, $Ω¹_X$ being locally free, the reflexive hulls $𝒬_i^{**}$, which are
  automatically locally free, also inherit Higgs fields, say $τ_i^{**}$, which
  make $(𝒬_i^{**}, τ_i^{**})$ stable with respect to $N$.  Since $𝒬_i$ and
  $𝒬_i^{**}$ agree in codimension one, $c_1(𝒬_i) = c_1(𝒬_i^{**})$, and
  $c_1(𝒬_i^{**})·N = 0$.

  \subsubsection*{Step 3.  Proof in general -- the Chern character of $𝒬_i^{**}$}
  \approvals{Behrouz & yes \\ Daniel & yes \\ Stefan & yes \\ Thomas & yes}

  We will prove in this step that
  \begin{equation}\label{eq:vdb}
    ch_2(𝒬_i^{**}) = ch_2(𝒬_i) = 0 \quad\text{for all }i,
  \end{equation}
  following \cite[proof of Thm.~2]{MR1179076} closely.  In fact, since
  $c_2(𝒬_i^{**}) ≤ c_2(𝒬_i)$, cf.~\cite[p.~80]{MR2665168}, we have the inverse
  inequality for the second Chern character, $ch_2(𝒬_i) ≤ ch_2(𝒬_i^{**})$.
  Additivity of Chern characters then implies that
  \begin{equation}\label{eq:ch-2x}
    0 = ch_2(ℰ)= \sum ch_2(𝒬_i) ≤ \sum ch_2(𝒬_i^{**}).
  \end{equation}
  On the other hand, the Bogomolov-Gieseker inequality,
  \cite[Thm.~2.1]{MR1954067} or \cite[Thm.~6.1]{GKPT15}, give that
  \begin{equation}\label{eq:ch-3x}
    2·ch_2(𝒬_i^{**}) ≤ \frac{1}{\rank 𝒬_i}·c_1(𝒬_i^{**})².
  \end{equation}
  As before, equality $c_1(𝒬_i^{**})·N = 0$ together with Hodge index theorem
  implies that $c_1(𝒬_i^{**})² ≤ 0$, so Inequality~\eqref{eq:ch-3x} reduces to
  $ch_2(𝒬_i^{**})≤ 0$.  Hence \eqref{eq:ch-2x} implies \eqref{eq:vdb}.

  \subsubsection*{Step 4.  Proof in general -- end of proof}
  \approvals{Behrouz & yes \\ Daniel & yes \\ Stefan & yes \\ Thomas & yes}

  Equation \eqref{eq:vdb} allows us to apply the results of Step~1 to the Higgs
  bundles $(𝒬_i^{**}, τ_i^{**})$.  This implies in particular that all Chern
  classes of $𝒬_i^{**}$ vanish, so that the sheaves $𝒬_i$ and $𝒬_i^{**}$ agree
  to start with.  The Higgs sheaf $(ℰ,θ)$ is thus (an iterated) extension of
  Higgs bundles with vanishing Chern classes that are stable with respect to any
  ample polarisation.  Item~\ref{il:y2} follows, which ends the proof of
  Proposition~\ref{prop:sttc} and therefore the proof of
  Theorem~\ref*{thm:bphgs} in dimension $2$.
\end{proof}

\subsection{Proof of Theorem~\ref*{thm:bphgs} if $X$ is maximally quasi-étale}
\approvals{Behrouz & yes \\ Daniel & yes \\ Stefan & yes \\ Thomas & yes}
\label{ssec:potixmq}

We will now prove Theorem~\ref{thm:bphgs} under the additional assumption that
the natural push-forward morphism of algebraic fundamental groups,
$\what{π}_1(X_{\reg}) → \what{π}_1(X)$, is isomorphic.  Following \cite{GKP13},
we say shortly that $X$ is maximally quasi-étale.  This assumption will be
maintained throughout the present Section~\ref{ssec:potixmq}.  Choose a divisor
$Δ$ such that $(X,Δ)$ is klt.

\begin{notation}
  Choose an ample divisor $H ∈ \Div(X)$ such that $(ℰ,θ)$ is semistable with
  respect to $H$ and satisfies $ch_1(ℰ)·[H]^{n-1} = ch_2(ℰ)·[H]^{n-2} = 0$.  Set
  $(\wtilde{ℰ},\wtilde{θ}) := π^*(ℰ,θ)$ and $\wtilde{H} := π^*H$.  If $A$ is a
  subvariety of $B$, we denote the obvious inclusion morphism by $ι_A$.
\end{notation}

Finally, let $\aF$ be the family of locally free Higgs sheaves $(ℱ,Θ)$ on $X$,
that satisfy $\rank ℱ = \rank ℰ$ and have the additional property that
$π^*(ℱ, Θ) ∈ \Higgs_{\wtilde{X}}$.  Recall from Fact~\ref{fact:dfsj} that this
family is bounded.

\subsubsection*{Step 1: Choice of a complete intersection surface}
\CounterStep
\approvals{Behrouz & yes \\ Daniel & yes \\ Stefan & yes \\ Thomas & yes}

Our proof relies on the choice of a sufficiently general complete intersection
surface $S ⊂ X$ to which we can restrict.  To be precise, choosing a
sufficiently increasing sequence of numbers $0 ≪ m_1 ≪ m_2 ≪ ⋯ ≪ m_{n-2}$ as
well as a sufficiently general tuple of hyperplanes,
$$
(D_1, …, D_{n-2}) ∈ |m_1·H|⨯ ⋯ ⨯ |m_{n-2}·H|,
$$
the following will hold.
\begin{enumerate}
\item\label{item:Bertini} The intersection $S := D_1 ∩ ⋯ ∩ D_{n-2}$ is an
  irreducible, normal surface, not contained in any component of $Δ$, and the
  pair $(S,Δ|_S)$ is klt ---Seidenberg's theorem \cite[Thm.~1.7.1]{BS95} and
  \cite[Lem.~5.17]{KM98}.
  
\item\label{item:stable} The restriction $(ℰ,θ)|_{S}$ is semistable with respect
  to the ample divisor $H|_S$ ---Restriction theorem,
  Theorem~\ref{thm:restrictionA}.
  
\item\label{item:pi1} The natural morphism $i_*: π_1(S_{\reg}) → π_1(X_{\reg})$,
  induced by the inclusion $i: S_{\reg} → X_{\reg}$, is isomorphic ---Lefschetz
  hyperplane theorem for fundamental groups, \cite[Thm.\ in
  Sect.~II.1.2]{GoreskyMacPherson}.
  
\item\label{item:sheafisom1} Given any Higgs sheaf $(ℱ, Θ) ∈ \aF$, then
  $(ℰ,θ) ≅ (ℱ, Θ)$ if and only if $(ℰ,θ)|_{S} ≅ (ℱ, Θ)|_{S}$ ---Iterated
  Bertini theorem, Proposition~\ref{prop:restr}.
\end{enumerate}

Let $\wtilde{S} ⊂ \wtilde{X}$ be the strict transform of $S$ in $\wtilde{X}$.
Then $\wtilde{S}$ is a smooth surface and $π_{{S}}: \wtilde{S} → S$ is a
resolution.  The following diagram summarises the situation.
\begin{equation}\label{eq:resolutiondiagram}
  \begin{aligned}
    \xymatrix{ %
      \wtilde{S} \ar@{^(->}[rr]^{ι_{\wtilde{S}}} \ar[d]_{π_S} && \wtilde{X} \ar[d]^{π} \\
      S \ar[rr]_{ι_S} && X.  %
    }
  \end{aligned}
\end{equation}

\begin{claim}\label{claim:ghd}
  The natural morphism of étale fundamental groups,
  $\what{π}_1(\wtilde{S}) → \what{π}_1(\wtilde{X})$ is isomorphic.  In
  particular every local system on $\wtilde{S}$ is the restriction of a local
  system on $\wtilde{X}$.
\end{claim}
\begin{proof}
  Consider the following diagram of push-forward morphisms of étale fundamental
  groups,
  $$
  \xymatrix{ %
    \what{π}_1(\wtilde{S}) \ar[rr]^{(ι_{\wtilde{S}})_*} \ar[d]_{α_S} && \what{π}_1(\wtilde{X}) \ar[d]^{α_X} \\
    \what{π}_1(S) \ar[rr]_{(ι_S)_*} && \what{π}_1(X) \\
    \what{π}_1(S_{\reg}) \ar[rr]_{(ι_{S_{\reg}})_*} \ar[u]^{β_S} && \what{π}_1(X_{\reg}) \ar[u]_{β_X}
  }
  $$
  The group morphisms $α_S$ and $α_X$ are isomorphisms by Takayama,
  \cite{Takayama2003}.  The morphism $β_S$ is surjective by e.g.\ \cite[0.7.B on
  p.~33]{FL81}, and by the fact the profinite completion is right exact,
  \cite[Lem.~3.2.3 and Prop.~3.2.5]{RB10}.  The morphism $β_X$ is isomorphic by
  assumption, and $(ι_{S_{\reg}})_*$ is isomorphic by Item~\ref{item:pi1}.  The
  remaining arrows $(ι_{\wtilde{S}})_*$ and $(ι_S)_*$ must then also be
  isomorphic.  The asserted extension of local systems comes from the fact that
  representation of $π_1(S)$ comes from a representation of $π_1(X)$,
  cf.~\cite[Thm.~1.2b]{MR0262386}, or \cite[Sect.~8.1]{GKP13} for a detailed
  pedestrian proof
\end{proof}

\subsubsection*{Step 2: End of proof}
\approvals{Behrouz & yes \\ Daniel & yes \\ Stefan & yes \\ Thomas & yes}

We consider the restriction of $(\wtilde{ℰ},\wtilde{θ})$ to $\wtilde{S}$.
Commutativity of Diagram~\eqref{eq:resolutiondiagram} and the fact that all
spaces involved are klt allow us to apply \cite[Lem.~5.9]{GKPT15} to conclude
that
\begin{equation}\label{eq:restriction_commutes}
  (\wtilde{ℰ},\wtilde{θ})|_{\wtilde S} ≅ π_{S}^* \bigl( (ℰ, θ)|_S \bigr),
\end{equation}
which is hence semistable with respect to $π_{S}^*(H|_S)$, thanks to
Item~\ref{item:stable} above and \cite[Prop.~5.19]{GKPT15}.
Proposition~\ref{prop:sttc} therefore implies that
$(\wtilde{ℰ},\wtilde{θ})|_{\wtilde{S}} ∈ \Higgs_{\wtilde{S}}$.  The classical
nonabelian Hodge correspondence applies, and gives a local system
$\aE_{\wtilde{S}} ∈ \LSys_{\wtilde{S}}$.  According to Claim~\ref{claim:ghd}, we
find a local system $\aE_{\wtilde{X}} ∈ \LSys_{\wtilde{X}}$ whose restriction
$\aE_{\wtilde{X}}|_{\wtilde{S}}$ is isomorphic to $\aE_{\wtilde{S}}$.  The
classical nonabelian Hodge correspondence on $\wtilde{X}$ thus yields a Higgs
sheaf $(\wtilde{ℱ}, \wtilde{Θ}) ∈ \Higgs_{\wtilde{X}}$ with vanishing Chern
classes, whose restriction $(\wtilde{ℱ}, \wtilde{Θ})|_{\wtilde{S}}$ is
isomorphic to $(\wtilde{ℰ},\wtilde{θ})|_{\wtilde{S}}$ by functoriality.  It
follows from Proposition~\ref{prop:Higgs_are_num_flat} that $\wtilde{ℱ}$ is
$π$-numerically flat.  Item~\ref{il:desc2} of Theorem~\ref{thm:desc2} therefore
yields a locally free Higgs sheaf $(ℱ, Θ) ∈ \aF$ whose restriction $(ℱ, Θ)|_S$
is isomorphic to $(ℰ,θ)|_S$ owing to \eqref{eq:restriction_commutes}.  In
particular, Item~\ref{item:sheafisom1} applies to show that
$(ℰ,θ) ≅ (ℱ, Θ)$.  By construction,
$π^*(ℰ,θ) ≅ (\wtilde{ℱ}, \wtilde{Θ}) ∈ \Higgs_{\wtilde{X}}$, which concludes
the proof of Theorem~\ref{thm:bphgs} in case where $X$ is maximally quasi-étale.
\qed

\subsection{Proof of Theorem~\ref*{thm:bphgs} in the general setting}
\label{ssec:gydydg}
\approvals{Behrouz & yes \\ Daniel & yes \\ Stefan & yes \\ Thomas & yes}

Finally, we prove Theorem~\ref{thm:bphgs} (``Ascent of semistable Higgs
bundles'') without any additional assumptions.  Choose an ample divisor
$H ∈ \Div(X)$ such that $(ℰ,θ)$ is semistable with respect to $H$ and satisfies
$ch_1(ℰ)·[H]^{n-1} = ch_2(ℰ)·[H]^{n-2} = 0$ and write
$(\wtilde{ℰ},\wtilde{θ}) := π^*(ℰ,θ)$.

Recall from \cite[Thm.~1.5]{GKP13} that there exists a quasi-étale cover
$f: Y → X$ such that $\what{π}_1(Y_{\reg}) ≅ \what{π}_1(Y)$.  Choose one
such $f$, note that the corresponding $Y$ is again klt, and let $\wtilde{Y}$ be
a desingularisation of the (unique) irreducible component of the fibre product
$Y ⨯_X \wtilde{X}$ that dominates both $Y$ and $\wtilde{X}$.  We obtain a
diagram of surjections as follows,
$$
\xymatrix{ %
  \wtilde{Y} \ar[rrr]^{\wtilde{f}\text{, gen.\ finite}} \ar[d]_{\wtilde{π}\text{, resolution}} &&& \wtilde{X} \ar[d]^{π\text{, resolution}} \\
  Y \ar[rrr]_{f\text{, quasi-étale}} &&& X.
}
$$

Recall from Corollary~\ref{cor:suqema} that the Higgs sheaf
$(ℰ_Y,θ_Y) := f^*(ℰ,θ)$ is semistable with respect to the ample divisor
$H_Y:= f^*H$.  It clearly satisfies
$ch_1(ℰ_Y)·[H_Y]^{n-1} = ch_2(ℰ_Y)·[H_Y]^{n-2}=0$.  In particular,
$(ℰ_Y,θ_Y) ∈ \Higgs_Y$.  We have seen in Section~\ref{ssec:potixmq} that this
implies that $\wtilde{π}^*(ℰ_Y,θ_Y) ∈ \Higgs_{\wtilde{Y}}$.  Lemma~\ref{lem:xxs}
thus yields the following.
\begin{enumerate}
\item\label{il:b1} All Chern classes
  $c_i \bigl( \wtilde{π}^*\, ℰ_Y \bigr) = c_i \bigl( \wtilde{f}^* \wtilde{ℰ}
  \bigr) ∈ H^{2i}\bigl( \wtilde{Y},\, ℚ \bigr)$ vanish.

\item\label{il:b2} The Higgs sheaf $\wtilde{π}^*(ℰ_Y,θ_Y)$ is semistable with
  respect to \emph{any} nef divisor on $\wtilde{Y}$.
\end{enumerate}

Let $\wtilde{H} ∈ \Div(\wtilde{X})$ be any ample.  Item~\ref{il:b1} immediately
implies that
$$
ch_1(\wtilde{ℰ}) · [\wtilde{H}]^{n-1} = ch_2(\wtilde{ℰ})· [\wtilde{H}]^{n-2}=0.
$$
In a similar vein, Item~\ref{il:b2} and Corollary~\ref{cor:suqema} imply that
$(\wtilde{ℰ},\wtilde{θ})$ is semistable with respect to $\wtilde{H}$.  This
concludes the proof of Theorem~\ref{thm:bphgs}.  \qed

%
% Do not edit the following line.  The text is automatically updated by
% subversion.
%
\svnid{$Id: 08.tex 1145 2018-04-25 11:46:56Z kebekus $}

\section{Proof of the nonabelian Hodge correspondence for klt spaces}\label{sect:07}
\subversionInfo

\subsection{Proof of Theorem~\ref*{thm:free-naht} (``Nonabelian Hodge correspondence for klt spaces'')}
\approvals{Behrouz & yes \\ Daniel & yes \\ Stefan & yes \\ Thomas & yes}
\label{sec:pfnh2}

Using the results on descent and ascent for Higgs bundles,
Theorems~\ref{thm:desc2} and \ref{thm:bphgs}, we construct the relevant functors
between $\Higgs_X$ and $\LSys_X$.  Once the functors are constructed, we show
that they indeed satisfy all the claims made in Theorem~\ref{thm:free-naht}.  We
maintain assumptions and notation of Theorem~\ref{thm:free-naht} throughout
Section~\ref{sec:pfnh2}.  The canonical resolution of singularities is denoted
by $π : \wtilde{X} → X$.

\subsubsection*{Step 1: from local systems to bundles}
\approvals{Behrouz & yes \\ Daniel & yes \\ Stefan & yes \\ Thomas & yes}

We establish the first half of Theorem~\ref{thm:free-naht}, constructing a
functor that maps local systems to Higgs bundles.  Given a local system
$\aE ∈ \LSys_X$, consider the Higgs bundle
$(\wtilde{ℰ}, \wtilde{θ}) := η_{\wtilde{X}}(π^* \aE)$ associated to $π^* \aE$
via the nonabelian Hodge correspondence on the manifold $\wtilde{X}$.  Recall
that $(\wtilde{ℰ}, \wtilde{θ}) ∈ \Higgs_{\wtilde{X}}$ and set
$ℰ := π_* \wtilde{ℰ}$.  Proposition~\ref{prop:Higgs_are_num_flat} and
Item~\ref{il:desc2} of Theorem~\ref{thm:desc2} imply that $ℰ$ is locally free
and that it carries a unique Higgs field $θ$ such that
$(\wtilde{ℰ},\wtilde{θ}) ≅ π^*(ℰ,θ)$.  Lemma~\ref{lem:xxs} applies to show
that $(ℰ,θ) ∈ \Higgs_X$.  In summary, we have constructed a mapping
$η_X : \LSys_X → \Higgs_X$.

If $e : \aE_1 → \aE_2$ is a morphism of local systems, we obtain a morphism
$\wtilde{e} : π^*(\aE_1) → π^*(\aE_2)$ and hence a morphism between the
associated Higgs sheaves,
$η_{\wtilde{X}}(\wtilde{e}) : (\wtilde{ℰ}_1,\wtilde{θ}_1) →
(\wtilde{ℰ}_2,\wtilde{θ}_2)$.  Denoting the associated, locally free Higgs
sheaves on $X$ by $(ℰ_1,θ_1)$ and $(ℰ_2,θ_2)$, an elementary computation shows
that $η_{\wtilde{X}}(\wtilde{e})$ descends to a morphism
$η_X(e) : (ℰ_1,θ_1) → (ℰ_2,θ_2)$.  In other words, we constructed a functor
$η_X : \LSys_X → \Higgs_X$.

\begin{obs}\label{obs:1f}
  If $X$ is smooth, then the canonical resolution of singularities is the
  identity, and $η_X$ equals the functor given by the nonabelian Hodge
  correspondence.
\end{obs}

\begin{obs}\label{obs:2f}
  The natural map $π^* ℰ = π^* π_*\, \wtilde{ℰ} → \wtilde{ℰ}$ induces an
  isomorphism of Higgs sheaves,
  $N_{π,\aE} : π^* η_X(\aE) → η_{\wtilde{X}}\bigl( π^* \aE \bigr)$.
\end{obs}

\subsubsection*{Step 2: from bundles to local systems}\label{ssec:bls}
\approvals{Behrouz & yes \\ Daniel & yes \\ Stefan & yes \\ Thomas & yes}

Given a Higgs bundle $(ℰ,θ) ∈ \Higgs_X$, consider the Higgs bundle
$(\wtilde{ℰ}, \wtilde{θ}) := π^*(ℰ,θ)$.  Theorem~\ref{thm:bphgs} asserts that
$(\wtilde{ℰ}, \wtilde{θ})$ is in $\Higgs_{\wtilde{X}}$.  Apply the classical
nonabelian Hodge correspondence on the manifold $\wtilde{X}$, in order to obtain
a local system
$\aE_{\wtilde{X}} := μ_{\wtilde{X}}\bigl( \wtilde{ℰ}, \wtilde{θ}) \bigr)$.  As
before, Takayama's result \cite[p.~827]{Takayama2003} shows that
$\aE := π_* \aE_{\wtilde{X}}$ is a local system on $X$.  We leave it to the
reader to show that this construction does indeed give a functor
$μ_X : \Higgs_X → \LSys_X$ and that the following observations hold.

\begin{obs}\label{obs:3f}
  If $X$ is smooth, then the canonical resolution of singularities is the
  identity, and $μ_X$ equals the functor given by the classical nonabelian Hodge
  correspondence.
\end{obs}

\begin{obs}\label{obs:4f}
  The natural map $π^* \aE = π^* π_*\, \aE_{\wtilde{X}} → \aE_{\wtilde{X}}$
  induces an isomorphism of local systems,
  $M_{π,(ℰ,θ)} : π^* μ_X(\aE) → μ_{\wtilde{X}}\bigl( π^* \aE \bigr)$.
\end{obs}

\subsubsection*{Step 3: equivalence of categories}
\approvals{Behrouz & yes \\ Daniel & yes \\ Stefan & yes \\ Thomas & yes}

Let $X$ be a projective klt space, and let $π : \wtilde{X} → X$ be the canonical
resolution.  The functors $η_{\wtilde{X}}$ and $μ_{\wtilde{X}}$ associated with
the classical nonabelian Hodge correspondence on the manifold $\wtilde{X}$ form
an equivalence of categories: the compositions $η_{\wtilde{X}} ◦ μ_{\wtilde{X}}$
and $μ_{\wtilde{X}} ◦ η_{\wtilde{X}}$ is naturally isomorphic to the identities
on $\Higgs_{\wtilde{X}}$ and $\LSys_{\wtilde{X}}$, respectively.  One checks
immediately that these isomorphisms descend to $X$, showing that the functors
$η_X$ and $μ_X$ constructed above do indeed give an equivalence of categories.

\subsubsection*{Step 4: end of proof}
\approvals{Behrouz & yes \\ Daniel & yes \\ Stefan & yes \\ Thomas & yes}

It remains to show Items~\ref{il:HC1}--\ref{il:HC3} of
Theorem~\ref{thm:free-naht}.  These, however, follow immediately from
Observations~\ref{obs:1f}--\ref{obs:4f}.  The proof of
Theorem~\ref{thm:free-naht} is thus complete.  \qed

\subsection{Proof of Theorem~\ref*{thm:finm} (``Functoriality in morphisms'')}\label{ssec:gghfj1}
\approvals{Behrouz & yes \\ Daniel & yes \\ Stefan & yes \\ Thomas & yes}

To keep the paper reasonably short, we will only consider functoriality of the
functors $η_•$.  Functoriality of $μ_•$ follows along the same lines of
argument.  We construct the isomorphisms $N_{•,•}$ in some detail, but leave the
tedious and lengthy verification of the construction's properties to the reader,
as none of the required arguments is in any way surprising or holds a promise of
new insight.

\subsubsection*{Step 1: lifting morphisms}
\approvals{Behrouz & yes \\ Daniel & yes \\ Stefan & yes \\ Thomas & yes}

For morphism between manifolds, functoriality of the nonabelian Hodge
correspondence is classically known, \cite[Rem.\ on p.~36]{MR1179076}.  If
$f : X → Y$ is a morphism between klt spaces that are not smooth, the following
claim allows to lift $f$ to a morphism between spaces that are smooth, though
possibly of higher dimension.

\begin{claim}\label{claim:i1}
  Given a morphism $f: X → Y$ of projective klt spaces, there exists a smooth,
  projective variety $\wtilde{W}$, and a commutative diagram as follows,
  \begin{equation}\label{eq:jvdf}
    \begin{gathered}
      \xymatrix{ %
        \wtilde{W} \ar[rrrrr]^{\wtilde{f}} \ar[d]_{Π\text{, conn.\ fibres}} &&&&& W ⊆ \wtilde{X} \ar@<3.5mm>[d]^{π\text{, canon.\ resolution}} \ar@<-3.5mm>[d]_{π_W\text{, conn.\ fibres}} \\
        \wtilde{Y} \ar[rrr]_{π_Y\text{, canon.\ resolution}} &&& Y \ar[rr]_{f} && \:Z\: ⊆ X,
      }
    \end{gathered}
  \end{equation}
  where $Z = f(Y)$ and $W = \wtilde{f}(\wtilde{W})$.
\end{claim}
\begin{proof}
  Recall from \cite[Cor.~1.7(1)]{HMcK07} that there exists an irreducible
  variety $W ⊆ π^{-1}(Z) ⊆ \wtilde{X}$ such that the induced morphism
  $π_W : W → Z$ is surjective, and such that its general fibre is rationally
  connected, and in particular irreducible.  Choose one such $W$.  Consider the
  fibre product $W ⨯_Z \wtilde{Y}$, choose the unique component that dominates
  both $W$ and $Y$, and let $\wtilde{W}$ its canonical resolution.  The general
  fibre of $Π$ is again rationally connected and irreducible, in particular
  connected.
\end{proof}

\subsubsection*{Step 2: Construction of $N_{•,•}$}
\approvals{Behrouz & yes \\ Daniel & yes \\ Stefan & yes \\ Thomas & yes}

With the help of Claim~\ref{claim:i1}, we construct the isomorphisms $N_{•,•}$
in this section.

\begin{construction}\label{cons:kl1}
  Assume we are given a morphism $f: X → Y$ of projective klt spaces and a local
  system $\aE ∈ \LSys_X$.  Choose a diagram $(\aD)$ as in Claim~\ref{claim:i1}
  and construct an isomorphism
  $$
  α_{\aD} : Π^* π_Y^* f^* \, η_X(\aE) → Π^* π_Y^* \, η_Y \bigl( f^* \aE \bigr)
  $$
  as the composition of the following canonical isomorphisms,
  \begin{align*}
    Π^* π_Y^* f^*\, η_X(\aE) & ≅ \wtilde{f}^* π^* \, η_X(\aE) && \text{Commutativity of (\aD)} \\
                       & ≅ \wtilde{f}^* \, η_{\wtilde{X}} \bigl(π^*\, \aE \bigr) && \text{Item~\ref{il:HC3} of Theorem~\ref{thm:free-naht}} \\
                       & ≅ η_{\wtilde{W}} \bigl(\wtilde{f}^* π^*\, \aE \bigr) && \text{Simpson's functoriality for }\wtilde{f}: \wtilde{W} → \wtilde{X} \\
                       & ≅ η_{\wtilde{W}} \bigl(Π^* π_Y^* f^*\, \aE \bigr) && \text{Commutativity of \eqref{eq:jvdf}} \\
                       & ≅ Π^*\, η_{\wtilde{Y}} \bigl(π_Y^* f^*\, \aE \bigr) && \text{Simpson's functoriality for } Π: \wtilde{W} → \wtilde{Y} \\
                       & ≅ Π^*π_Y^*\, η_{Y} \bigl(f^*\, \aE \bigr) && \text{Item~\ref{il:HC3} of Theorem~\ref{thm:free-naht}.}
  \end{align*}
  Since $π_Y◦Π$ has connected fibres, the isomorphism $α_{\aD}$ descends to an
  isomorphism
  $$
  β_{\aD} : f^*η_X(\aE) → η_Y \bigl(f^* \aE \bigr).
  $$
\end{construction}

\begin{claim}
  In the setting of Construction~\ref{cons:kl1}, the morphism $β_{\aD}$ is
  independent of the choices made.  More precisely, given two diagrams $\aD_1$
  and $\aD_2$ as in Claim~\ref{claim:i1}, then $β_{\aD_1} = β_{\aD_2}$.  We can
  therefore choose $\aD$ arbitrarily and set $N_{f,\aE} := β_{\aD}$.
\end{claim}
\begin{proof}
  Left to the reader.
\end{proof}

\subsubsection*{Step 3: end of proof}
\approvals{Behrouz & yes \\ Daniel & yes \\ Stefan & yes \\ Thomas & yes}

The isomorphisms $N_{•,•}$ constructed in Step~2 clearly have the expected
\emph{behaviour under canonical resolution} and satisfy the \emph{compatibility
  conditions} spelled out in Theorem~\ref{thm:finm}.  We leave it to the reader
to verify \emph{functoriality} and to write down the analogous construction of
the morphisms $M_{•,•}$.  \qed

\subsection{Proof of Theorem~\ref*{thm:iop} (``Independence of polarisation'')}
\label{ssec:gghfj2}
\approvals{Behrouz & yes \\ Daniel & yes \\ Stefan & yes \\ Thomas & yes}

We establish the following sequence of implications.

\medskip % do not know why I need this medskip.  If I remove it, the formula runs into the text body

$$
\xymatrix{ %
  \text{\ref{il:q5}} \ar@{=>}[r]_{\text{trivial}} & \text{\ref{il:q4}} \ar@{=>}[r] & \text{\ref{il:q3}} \ar@{=>}[r]_{\text{trivial}} & \text{\ref{il:q2}} \ar@{=>}[r]_{\text{trivial}} & \text{\ref{il:q1}} \ar@/_.5cm/@{=>}[llll]
}
$$

\subsubsection*{Implication \ref{il:q4} $⇒$ \ref{il:q3} in the semistable case}
\approvals{Behrouz & yes \\ Daniel & yes \\ Stefan & yes \\ Thomas & yes}

This is given by Lemma~\ref{lem:xxs} (``Description of bundles whose pull-back
is in $\Higgs_{\wtilde{X}}$'').

\subsubsection*{Implication \ref{il:q4} $⇒$ \ref{il:q3} in the stable case}
\approvals{Behrouz & yes \\ Daniel & yes \\ Stefan & yes \\ Thomas & yes }

Choose $π: \wtilde{X} → X$ and $\wtilde{H} ∈ \Div(\wtilde{X})$ as in
\ref{il:q4}.  Vanishing of Chern classes has been shown in the semistable case.
As for stability, let $H ∈ \Div(X)$ be any ample divisor.  Choose a sufficiently
large number $m ≫ 0$ such that $|m·H|$ is basepoint free, choose sufficiently
general hyperplanes $H_1, …, H_{\dim X-1} ∈ |m·H|$ and consider the associated
complete intersection curve $C := H_1 ∩ ⋯ ∩ H_{\dim X-1} ⊂ X$.  Observe that $C$
is entirely contained in $X_{\reg}$, and that $π$ is isomorphic near $C$.
Denote the preimage by $\wtilde{C} := π^{-1}C$.  We obtain the following obvious
morphisms between fundamental groups,
$$
\xymatrix{ %
  π_1(\wtilde{C}) \ar[rrrrrr]^{\wtilde{i}_*} \ar[d]_{\text{isomorphic}} &&&&&& π_1(\wtilde{X}) \ar[d]^{\txt{\scriptsize isomorphic\\\scriptsize \cite[Thm.~1.1]{Takayama2003}}} \\
  π_1(C) \ar@{->>}[rrr]^{\text{surjective}}_{\text{\cite[Thm.\ in Sect.\ II.1.2]{GoreskyMacPherson}}} &&& π_1(X_{\reg}) \ar@{->>}[rrr]^{\text{surjective}}_{\text{\cite[0.7.B on p.~33]{FL81}}} &&& π_1(X).
}
$$
In particular, we see that the morphism $\wtilde{i}_*$ is surjective.  Apply the
nonabelian Hodge correspondence for klt spaces, Theorem~\ref{thm:free-naht} to
obtain local systems $\aE$ on $X$ and
$π^*\, \aE ≅ μ_{\wtilde{X}}\bigl(π^*\,(ℰ,θ)\bigr)$ on $\wtilde{X}$.  The
following will then end the proof.
\begin{align*}
  π^*\,(ℰ,θ) \text{ is stable w.r.t.\ } \wtilde{H} & ⇒ π^*\, \aE \text{ is irreducible } && \text{\cite[p.~18ff]{MR1179076}} \\
                                                   & ⇒ (π^* \aE)|_{\wtilde{C}} \text{ is irreducible } && \text{surjectivity of }\wtilde{i}_*\\
                                                   & ⇒ \aE|_C \text{ is irreducible} && \text{isomorphic} \\
                                                   & ⇒ (ℰ,θ)|_C \text{ is stable} && \text{\cite[p.~18ff]{MR1179076}}\\
                                                   & ⇒ \text{$(ℰ,θ)$ is stable w.r.t.\ $H$} && \text{genl.\ choice}
\end{align*}

\subsubsection*{Implication \ref{il:q4} $⇒$ \ref{il:q3} in the polystable case}
\approvals{Behrouz & yes \\ Daniel & yes \\ Stefan & yes \\ Thomas & yes}

Analogous to the above, with ``irreducible'' replaced by ``semisimple''.

\subsubsection*{Implication \ref{il:q1} $⇒$ \ref{il:q5} in the semistable case}
\approvals{Behrouz & yes \\ Daniel & yes \\ Stefan & yes \\ Thomas & yes}

This is Theorem~\ref{thm:bphgs} (``Ascent of semistable Higgs bundles'')
together with \cite[Thm.~1.3]{MR2231055}.

\subsubsection*{Implication \ref{il:q1} $⇒$ \ref{il:q5} in the stable case}
\approvals{Behrouz & yes \\ Daniel & yes \\ Stefan & yes \\ Thomas & yes}

The pull-back bundle $π^*\,(ℰ,θ)$ is clearly stable with respect to the nef
bundle $π^*H$, \cite[Prop.~5.19]{GKPT15}.  But then openness of stability,
\cite[Prop.~4.17]{GKPT15} asserts that $π^*\,(ℰ,θ)$ is stable with respect to
one ample, hence any ample bundle.

\subsubsection*{Implication \ref{il:q1} $⇒$ \ref{il:q5} in the polystable case}
\approvals{Behrouz & yes \\ Daniel & yes \\ Stefan & yes \\ Thomas & yes}

Same as the stable case.  \qed

\end{document}